\documentclass[12pt,a4paper]{article}
\usepackage{a4wide,amsfonts,amsmath,latexsym,amsthm,amssymb,euscript,eufrak,graphicx,units,mathrsfs,setspace,stmaryrd,dsfont}

\usepackage[usenames,dvipsnames,table]{xcolor}

\definecolor{darkblue}{rgb}{0.3,0.3,0.7}

\usepackage[colorlinks=true, urlcolor=blue,linkcolor=blue, citecolor=blue]{hyperref}

\definecolor{ForestGreen}{RGB}{34,139,34}
\definecolor{mauve}{rgb}{0.7,0,0.43}
\definecolor{dkgreen}{rgb}{0,0.6,0}
\definecolor{darkgreen}{rgb}{0,0.6,0}
\definecolor{darkorange}{rgb}{1.0, 0.55, 0.0}
\definecolor{lightblue}{rgb}{0,0.2,0.5}
\definecolor{blue1}{rgb}{0,0.1,0.9}

\definecolor{lightblue}{rgb}{0,0.2,0.5}

\usepackage{systeme}
\usepackage{float}
\newfloat{figure}{H}{lof}
\floatname{figure}{\figurename}

\usepackage[french,english]{babel}
\usepackage[T1]{fontenc}
\usepackage[utf8]{inputenc}

\DeclareMathAlphabet{\eufrak}{U}{}{}{}  % Euler fraktur math
\SetMathAlphabet\eufrak{normal}{U}{euf}{m}{n}
\SetMathAlphabet\eufrak{bold}{U}{euf}{b}{n}

\newtheorem{prop}{Proposition}[section]

\newtheorem{theorem}[prop]{Theorem}
\newtheorem{lemma}[prop]{Lemma}

\newtheorem{assumption}[prop]{Assumption}

\newtheorem{remark}[prop]{Remark}

\newtheorem{definition}[prop]{Definition}
\newtheorem{notation}[prop]{Notation}

\numberwithin{equation}{section}

\def\E{\mathbb{E}}
\def\P{\mathbb{P}}
\def\real{\mathbb{R}}

\def\F{\mathcal{F}}
\def\1{\textbf{1}}

\newcommand{\eps}{\varepsilon}

\def\E{\mathbb E}
\def\F{{\cal F}}

\def\P{\mathbb P}

\makeatletter
\newcommand*\rel@kern[1]{\kern#1\dimexpr\macc@kerna}
\newcommand*\widebar[1]{%
  \begingroup
  \def\mathaccent##1##2{%
    \rel@kern{0.8}%
    \overline{\rel@kern{-0.8}\macc@nucleus\rel@kern{0.2}}%
    \rel@kern{-0.2}%
  }%
  \macc@depth\@ne
  \let\math@bgroup\@empty \let\math@egroup\macc@set@skewchar
  \mathsurround\z@ \frozen@everymath{\mathgroup\macc@group\relax}%
  \macc@set@skewchar\relax
  \let\mathaccentV\macc@nested@a
  \macc@nested@a\relax111{#1}%
  \endgroup
}
\makeatother

\def\cal#1{\mathcal{#1}}

\DeclareSymbolFontAlphabet{\mathrsfs}{rsfs}

\author{
Mahmoud Khabou\footnote{INSA de Toulouse, IMT UMR CNRS 5219, Universit\'e de Toulouse, 135 avenue de Rangueil 31077 Toulouse Cedex 4 France. \; Email: \texttt{mahmoud.khabou@insa-toulouse.fr}}
\and Nicolas Privault \footnote{School of Physical and Mathematical Sciences, Nanyang Technological University, 21 Nanyang Link, 637371 Singapre. \; Email: \texttt{nprivault@ntu.edu.sg}}
\and Anthony R\'eveillac\footnote{INSA de Toulouse, IMT UMR CNRS 5219, Universit\'e de Toulouse, 135 avenue de Rangueil 31077 Toulouse Cedex 4 France. \; Email: \texttt{anthony.reveillac@insa-toulouse.fr}}
}

\title{\huge Normal approximation of compound Hawkes functionals}

\begin{document}

\maketitle

\allowdisplaybreaks

\baselineskip0.6cm

\vspace{-0.6cm}

\begin{abstract}
  We derive quantitative bounds in the Wasserstein distance
  for the approximation of stochastic integrals with respect to
  Hawkes processes by a normally distributed random variable. 
  In the case of deterministic and non-negative integrands,
  our estimates involve only the third moment of
  integrand in addition to a variance term
  using a square norm of the integrand.
  As a consequence, we are able to observe a ``third moment phenomenon''
  in which the vanishing of the first cumulant can lead to faster
  convergence rates. 
  Our results are also applied to
  compound Hawkes processes,
  and improve on the current literature
  where estimates may not converge to zero in
  large time,
  or have been obtained only for specific
  kernels such as the exponential or Erlang kernels.
\end{abstract}

\noindent {\bf Key words:} 
Hawkes processes;
Stein method;
normal approximation;
Malliavin calculus. 
\\
  {\em Mathematics Subject Classification (2020):}
  60G55,
  60F05,
  60G57,
  60H07.
\baselineskip0.65cm
 
\section{Introduction}
Nourdin and Peccati \cite{Nourdin_Peccati} opened the way to a new methodology mixing Stein's method and the Malliavin calculus, to provide bounds on the distance between the distribution of a functional of a Gaussian field and a target Gaussian distribution. This analysis, which relies on a specific Gaussian structure, has been successfully transferred to the Gaussian approximation of Poisson functionals in \cite{Peccati_2010}.
Since then, several developments of the initial result of \cite{Peccati_2010} have been
obtained, for instance, Stein Gaussian approximation bounds have been obtained in \cite{Privault} in terms of the third cumulants for Poisson functionals
 expressed as the divergence of an adapted process with respect to an homogeneous Poisson process.
 Another development important for our analysis has been presented in \cite{torrisi} for counting processes with stochastic intensity (including the Hawkes process) using the so-called Poisson imbedding representation, see \cite{Bremaud_Massoulie}.
 This technique, also known as the ``Thinning Algorithm'', 
 allows one to represent a counting counting process and
 its intensity process as a solution to an SDE driven by an auxiliary Poisson
 random measure, and
 to adapted the general methodology of \cite{Peccati_2010}
 to this framework.
 Following here the path of \cite{torrisi}
 in the case of a linear Hawkes process $H$,
 i.e. a counting process with intensity process $\lambda:=(\lambda_t)_{t\geq 0}$ given as 
$$ \lambda_t = \mu +\int_{(0,t)} \phi(t-s) dH_s, $$
 with $\mu >0$, $\phi : \real_+\to \real_+$ and $\|\phi\|_1 < 1$, 
 a specific Malliavin calculus for Hawkes processes have been developed in \cite{Hillairet_Reveillac_Rosenbaum,HHKR},
 based on Relation~\eqref{eq:IBP} below
 on the simplification of the Malliavin integration by parts. 
 As a consequence, a new bound has been obtained in \cite{HHKR} for the Gaussian approximation of Hawkes functionals. 

 \medskip

 In \cite{Bacry_et_al_2013}, a functional convergence result
 has been obtained 
 for linear Hawkes processes, implying the (non-quantitative)
 convergence in distribution 
\begin{equation}
\nonumber % \label{eq:convintro}
F_T:=\frac{H_{T} - \int_0^{T} \lambda_s ds}{\sqrt{T}} \underset{T\to+\infty}{\overset{\mathcal{L}}{\longrightarrow}} \mathcal{N}(0,{\sigma}^2),
\end{equation}
 see Lemma~7 therein, 
 with $\sigma^2:=\mu / ( 1-\|\phi\|_1 )$.

 \medskip

 In this paper,
 we propose to quantify this convergence in the Wasserstein distance. 
 In Theorem~3.1 of \cite{torrisi},
 see also Relations~(5.2) and (5.4) therein,
 the estimates 
 $$ d_W\left(
 \frac{H_{T} - \int_0^{T} \lambda_s ds}{\sqrt{T}}
 ,
 \mathcal{N}(0,{\sigma}^2)\right) \leq B_T, \ \!  \textrm{ with } \ \!  B_T \geq
 \sqrt{\frac{8}{\pi}}
 \|\phi\|_1
 \frac{ 2-\|\phi\|_1}{1-\|\phi\|_1}
 $$
 have been derived, however
 they do not converge to $0$ as $T$ tends to $+\infty$,
 see Remark~\ref{rem:comment1} below.
 Related bounds have been derived independently
 in \cite[Theorem~3.4]{HHKR} for Hawkes
 processes, with in particular 
\begin{equation}
\nonumber % \label{eq:temp1intro}
d_W\left(
\frac{H_{T} - \int_0^{T} \lambda_s ds}{\sqrt{T}}
,
\mathcal{N}(0,{\sigma}^2)\right) \leq \frac{C_{\mu,\phi}}{\sqrt{T}} + R_T,
\end{equation}
see Theorems~3.10 and 3.12 therein, where $R_T$ is an additional term involving the Malliavin derivative of $F_T$,
see also Proposition~\ref{prop:TempDH} below.

\medskip

 In case $\phi$ is the exponential kernel
 $\phi(x):=\alpha e^{-\beta x}$, $\alpha < \beta$,
or the Erlang kernel $\phi(x):= \alpha x e^{-\beta x}$, $\alpha < \beta^2$,
the remainder term $R_T$ can be bounded to obtain 
 more accurate bounds of the form  
\begin{equation}
\label{eq:temp2intro}
d_W\left(
\frac{H_{T} - \int_0^{T} \lambda_s ds}{\sqrt{T}}
,
\mathcal{N}(0,{\sigma}^2) \right) \leq \frac{\widetilde{C}_{\mu,\phi}}{\sqrt{T}}. 
\end{equation}
\noindent
In this paper, we extend those results by deriving bounds of the form \eqref{eq:temp2intro} for Hawkes processes with general
kernel $\phi$ satisfying the condition $\|\phi\|_1<1$,
see Theorem~\ref{th:main2}, 
where $C_{\mu,\phi}>0$ is a constant and $\sigma^2$ an explicit asymptotic variance depending on $\mu,\phi$.

\medskip

For this, in Theorem~\ref{th:main1} we improve the bounds of \cite{torrisi} and \cite{HHKR} for Hawkes functionals of the form
$\int_0^\infty z(t) (dH_t-\lambda_t dt)$, where $z(t)$ is a deterministic function. 
 In particular,
 in Theorem~\ref{th:main1}-$(ii)$ 
 we provide an estimate involving only the third moment of
 $\int_0^\infty z(t) (dH_t-\lambda_t dt)$
 when $z(t)$ is deterministic and non-negative,  
 and an estimate on a modified second moment of $z(t)$ where the Malliavin derivative
 is not involved, see \eqref{eq:rela2} and the
 discussion in Remark \ref{rem:comment2}.

 \medskip

 In addition, these results are presented for a generalization of the Hawkes process,
 that is the compound Hawkes process $S$ given by 
$$
 S_t := \sum_{i=1}^{H_t} X_i, \quad t \geq 0,
 $$
 where the random variables $(X_i)_i$ are independent and identically distributed
 (i.i.d.) square integrable random variables independent of $H$.
 Furthermore, in the spirit of Theorem~3.7 in \cite{ChenGolsteinShao}, in Theorem~\ref{th:fast} we obtain faster rates of convergence for the compound Hawkes process
   $S$
   when the third cumulant of the
   jump sizes
   on the random variables $(X_i)_i$
   vanishes
   in the framework of the ``third moment phenomenon'',
   see Section~4.8 of \cite{ChenGolsteinShao}. Finally, we provide an alternative version of our quantitative limit theorems as Theorem \ref{th:alternativ}.

   \medskip
   
We proceed as follows. First, in Section~\ref{section:preliminaires} we present the main elements of stochastic analysis on the Poisson space and we recall the approach developed in \cite{Hillairet_Reveillac_Rosenbaum,HHKR} regarding the linear Hawkes process. Main results are collected in Section~\ref{section:main}. Proofs and technical lemmata are collected in Section~\ref{section:proofs}.

\section{Notations and preliminaries}
\label{section:preliminaires}

For $E$ a topological space, we set $\mathcal B(E)$ the $\sigma$-algebra of Borel sets. We denote by $dt$ the Lebesgue measure on $(\real, \mathcal B(\real))$.

\subsection{Stochastic analysis on the Poisson space}
% \label{sec:Poisson}
The notation and definitions stated in this section can be found for instance in \cite{Picard_French_96} or \cite{Privault_LectureNotes}. 
Let $\nu$ be a Borel measure on $\real$ with $\nu(\real)=1$ and $\nu(\{0\})=0$,
and consider the space of configurations
$$ \Omega^N:=\left\{\omega^N=\sum_{i=1}^{n} \delta_{(t_{i},\theta_i,x_i)}, \ \!  0=t_0 < t_1 < \cdots < t_n, \ \!  (\theta_i,x_i) \in \real_+\times \real,  \ \!  n\in \mathbb{N}\cup\{+\infty\} \right\}.
$$
Each path of a counting process is represented as an element $\omega^N$ in $\Omega^N$ which is a $\mathbb N$-valued measure on $\mathbb{R}_+^2\times \real$. Let $\mathcal F_\infty^N$ be the $\sigma$-field associated to the vague topology on $\Omega^N$, and $\P^N$ the Poisson measure under which the counting process $N$ defined as 
$$ N([0,t]\times[0,\theta]\times(-\infty,y])(\omega):=\omega([0,t]\times[0,\theta]\times(-\infty,x]), \quad (t,\theta,x) \in \real_+^2\times \real,$$
    is an homogeneous Poisson process with intensity measure $dt \otimes d\theta \otimes \nu$, that is, for any $(t,\theta,x) \in [0,T]\times \real_+\times \real$, $N([0,t]\times[0,\theta]\times(-\infty,x])$ is a Poisson random variable with intensity $ \theta \, t \, \nu((-\infty,x])$.
        We also let $\mathbb F^N:=(\F_t^N)_{t\geq 0}$ denote the natural history of $N$,
        that is
        $$
        \mathcal{F}_t^N:=\sigma(N( \mathcal{T} \times B), \ \!  \mathcal{T} \subset \mathcal{B}([0,t]), \ \!  B \in \mathcal{B}(\real_+\times \real)),
        \qquad t\geq 0, 
        $$
 hence $\mathcal F_\infty^N$ coincides with $\lim_{t\to+\infty} \mathcal F_t^N$. The expectation with respect to $\P^N$ is denoted by $\E[\ \! \cdot \ \! ]$, and the conditional expectation knowing $\mathcal {F}^N_t$ is denoted by $\E_t[\ \! \cdot \ \! ]$. 

\medskip

% \noindent
Next, we introduce the stochastic integral with respect to the Poisson point process $N$,
which will be related to the gradient operator $D$ defined below. 
\begin{definition}
We set
\begin{equation}
\nonumber % \label{eq:pred}
\mathcal P^N_2:=\left\{\rho:=(\rho_{(t,\theta,x)})_{(t,\theta,x)\in \real_+^3}\ \!  (\mathcal F^N_t)_{t\geq 0}\textrm{-predictable with } \int_{\real_+^3} |\rho_{(t,\theta,x)}|^2 dtd\theta \nu(dx)<+\infty\right\}.
\end{equation}
For $\rho:=(\rho_{(t,\theta,x)})_{(t,\theta,x)\in \real_+^3} \in \mathcal P^N_2$, we set the divergence operator
$$ \delta^N(\rho):=\int_{\real_+^2\times \real} \rho_{(t,\theta,x)} \left(N(dt,d\theta,dx)-dt d \theta \nu(dx)\right).$$
\end{definition}
\noindent
 Next, we introduce the Malliavin derivative first with respect to the Poisson point process $N$, using the adding point operator defined below.
 For any $A \in \mathcal{B}(\mathbb{R}_+\times \real_+\times \real)$, we let 
$$ \textbf{1}_A (t,\theta,x):=\left\lbrace \begin{array}{l} 1, \quad \textrm{if } (t,\theta,x)\in A, \\0, \quad \textrm{otherwise.}\end{array}\right. $$

\begin{definition}[Adding point operator]
% \label{definition:shift}
We define for $(t,\theta,x)$ in $\mathbb{R}_+\times \real_+\times \real$ the measurable maps 
$$ 
\begin{array}{lll}
\eps_{(t,\theta,x)}^+ : &\Omega^N \to \Omega^N\\
&\omega \mapsto  \eps_{(t,\theta,x)}^+(\omega),
\end{array}
$$
where we let 
$$(\eps_{(t,\theta,x)}^+(\omega))(A) := \omega(A \setminus {(t,\theta,x)}) + \textbf{1}_A(t,\theta,x),
\quad
A \in \mathcal{B}(\mathbb{R}_+\times \real_+\times \real). 
$$
\end{definition}
% \begin{lemma}
% \label{lemma:mesur}
\noindent
 We note that for any 
 $\mathcal F_t^N$-measurable random variable $F$, $t\geq 0$, we have
  $$ F\circ\eps_{(v,\theta,x)}^+ = F, \quad \P-a.s., 
  $$
 for all $v>t$ and $(\theta,x)\in \real_+\times \real$. 
% \end{lemma}

\begin{definition}[Malliavin derivative]
  For $F$ in $L^2\big(\Omega,\mathcal F_\infty^N,\P\big)$, we define
  the Malliavin derivative $D F$ of $F$ as 
$$ D_{(t,\theta,x)} F := F\circ \eps_{(t,\theta,x)}^+ - F, \quad (t,\theta,x) \in \real_+^2\times \real.  $$
\end{definition}
\noindent
We conclude this section with the integration by parts formula on the Poisson space,
 see \textit{e.g.} \cite{Picard_French_96} or \cite{Privault_LectureNotes}. 
\begin{prop}
% \label{prop:IPP}
Let $F$ be in $L^2\big(\Omega,\mathcal F_\infty^N,\P\big)$ and $\rho$ be in $\mathcal P^N_2$. We have that 
\begin{equation}
\label{eq:IBPPoisson}
\E\big[F \delta^N(\rho)\big] = \E\left[\int_{\real_+^2\times \real} \rho_{(t,\theta,x)} D_{(t,\theta,x)} F dt d\theta \nu(dx)\right].
\end{equation}
\end{prop} 

\subsection{Stochastic analysis for the compound Hawkes process}
% \label{section:MalliavinHawkes}
 We first recall the definition of a Hawkes process.
\begin{definition}[Standard Hawkes process, \cite{Hawkes}]
\label{def:standardHawkes}
Let $\mu>0$ and $\phi:\real_+ \to \real_+$ be a bounded non-negative function with $\|\phi\|_1 :=\int_0^\infty \phi(u) du<1$. The standard Hawkes process $H:=(H_t)_{t \geq 0}$ with parameters $\mu$ and $\phi$ is the counting process such that   
\begin{itemize}
\item[(i)] $H_0=0$, $\P$-$a.s.$,
\item[(ii)] its ($\mathbb{F}^N$-predictable) intensity process is given by
\begin{equation}
\label{eq:lambda}
\lambda_t:=\mu + \int_{(0,t)} \phi(t-s) dH_s, \quad t\geq 0,
\end{equation}
that is for any $0\leq s \leq t $ and $A \in \mathcal{F}^N_s$,
$$
\E\left[\textbf{1}_A (H_t-H_s) \right] = \E\left[\int_{(s,t]} \textbf{1}_A \lambda_r dr \right].
$$
\end{itemize}
\end{definition}
\noindent 
Note that the stochastic integral in \eqref{eq:lambda} is defined pathwise,
i.e. 
$$ \int_{(0,t)} \phi(t-s) dH_s := \sum_{0<s< t} \phi(t-s) \textbf{1}_{\{ \Delta_s H=1 \}},
$$
where the sum is well defined and finite $\P$-a.s. for every $t$,
 where we used the notation $\Delta_s H:=H_s-H_{s-}$.
\noindent This definition can be generalized as follows. 
\begin{definition}[Compound Hawkes process]
\label{definition:compound}
Consider a Hawkes process $H$ with parameters $\mu > 0$ and 
$\phi:\real_+ \to \real_+$ bounded non-negative with $\|\phi\|_1 <1$.
Given $(X_i)_{i\geq 1}$ an i.i.d sequence of random variables,
independent of $H$, with common distribution $\nu$,
 the process 
\begin{equation}
\label{eq:compound}
S_t := \sum_{i=1}^{H_t} X_i, \quad t \geq 0,
\end{equation}
 is called a compound Hawkes process. 
\end{definition}
\noindent 
Our approach uses the now classical construction of the Hawkes process
by ``thinning'' or ``Poisson embedding'' 
as the unique solution to an SDE with respect to 
 a Poisson random measure $N$, 
see \textit{e.g.} \cite{Ogata,Daley_VereJones,Bremaud_Massoulie,Costa_etal} and references therein.
We refer to \cite[Theorem~3.3]{Hillairet_Reveillac_Rosenbaum} for a precise statement on
 the uniqueness of solutions to the SDE~\eqref{eq:H}. 
 Here, we set $\mathbb F^H:=(\mathcal F_t^H)_{t \geq 0}$ (respectively $\mathbb F^S:=(\mathcal F_t^S)_{t \geq 0}$) the natural filtration of $H$ (respectively of $S$) and $\mathcal F_\infty^H:=\displaystyle \lim_{t\to+\infty} \mathcal F_t^H$ (respectively $\mathcal F_\infty^S:=\displaystyle \lim_{t\to+\infty} \mathcal F_t^S$),
  and we have $\mathcal F_t^H \subset \mathcal F_t^S \subset \mathcal F_t^N$ as $H$ is completely determined by the jump times of $H$, which coincide with those of $S$. 
\begin{theorem}[See Theorem~3.3 of \cite{Hillairet_Reveillac_Rosenbaum}]
% \label{th:HRR}
Let $\mu >0$ and $\phi:\real_+ \to \real_+$ such that $\|\phi\|_1<1$.
The system of stochastic differential equations 
\begin{equation}
\label{eq:H}
\left\lbrace
\begin{array}{l}
  S_t = \displaystyle{\int_{(0,t]\times \real_+\times \real} x \textbf{1}_{\{\theta \leq \lambda_s\}} N(ds,d\theta,dx)}, \quad t \geq 0,
    
\medskip

\\
H_t = \displaystyle{\int_{(0,t]\times \real_+\times \real} \textbf{1}_{\{\theta \leq \lambda_s\}} N(ds,d\theta,dx)},\quad t \geq 0,
  
\medskip

\\
\displaystyle \lambda_t = \mu + \int_{(0,t)} \phi(t-u) d H_u,\quad t \geq 0 .
\end{array}
\right.
\end{equation}
admits a unique solution $(X,H,\lambda)$ with $H$
(resp. $\lambda$) $\mathbb F^N$-adapted (resp. $\mathbb F^N$-predictable),
where 
$(H,\lambda)$ is a Hawkes process in the sense of Definition~\ref{def:standardHawkes},
 and $S$ is a compound Hawkes process in the sense of Definition~\ref{definition:compound}.
\end{theorem}
\noindent 
 We note that when $\nu(dx) = \delta_{1}(dx)$ equals the Dirac measure concentrated at $x=1$, i.e. $X_i \equiv 1$ in \eqref{eq:compound}, then $S \equiv H$.
\begin{notation}
We let $\mathcal{Z}:=(\mathcal{Z}_{(t,\theta)})_{(t,\theta) \in \real^2_+}$ denote the stochastic process defined as 
\begin{equation}
\label{eq:calZ}
\mathcal{Z}_{(t,\theta)}:= \textbf{1}_{\{\theta \leq \lambda_t\}}, \quad (t,\theta) \in \real_+^2.
\end{equation}
In this paper we consider
stochastic integrals 
\begin{align} 
\nonumber
F & =\delta^N(Z \mathcal{Z})
\\
\nonumber
& = \int_{\real_+^2\times \real} Z_{(t,x)} \mathcal{Z}_{(t,\theta)} \left(N(dt,d\theta,dx)-dt d \theta \nu(dx)\right)
\\
\nonumber % \label{eq:Fdivergenceform}
 & = \int_{\real_+\times \real} Z_{(t,x)} \left(dH_t-\lambda_tdt \nu(dx)\right) 
\end{align} 
 against the (compensated) Hawkes process, with $Z:=(Z_{(t,x)})_{(t,x) \in \real_+\times\real}$ an element of $\mathcal P^N_2$. 
 Most of our analysis will be carried out for a deterministic $Z$.
\end{notation}
We now specify the Malliavin derivative and the integration by parts formula \eqref{eq:IBPPoisson} for functionals of the Hawkes process (the Hawkes itself $H$ or the compound Hawkes process $X$).
For this, we note that by definition of $H$, any jump of $N$ at an atom $(t,\theta,x)$ turns out to be a jump of $H$ if and only if $\theta \in [0,\lambda_t]$, 
as stated in the next proposition. 
\begin{prop}[Proposition 2.16 \cite{HHKR}]
\label{prop:TempDH}
 For any $\mathcal F_\infty^H$-measurable random variable $F$ we have 
$$ D_{(t,0,x)} F = D_{(t,\theta,x)} F, \quad
\theta \in [0,\lambda_t], \ t \geq 0, \ x \in \real,
\ 
\P{\rm -}a.s..
$$
\end{prop}
\noindent 
In view of Proposition~\ref{prop:TempDH},
 for any $\mathcal F_\infty^H$-measurable random variable $F$, we set 
\begin{equation}
\label{def:DH}
 D_{(t,x)} F:= 
D_{(t,0,x)} F, \quad t \geq 0, \ x \in \real.
\end{equation} 
\noindent 
 Next, we state an integration by part formula for functionals of Hawkes processes. We recall below the integration by parts formula obtained in a particular case of \cite{HHKR} for the Hawkes process.
 
\begin{theorem}(Theorem~2.20 in \cite{HHKR})
Let $\mathcal{Z}:=(\mathcal{Z}_{(t,\theta)})_{(t,\theta) \in \real^2_+}$ be the stochastic process defined in \eqref{eq:calZ} and $Z:=(Z_{(t,x)})_{(t,x) \in \real_+ \times \real}$ be a $\mathbb F^N$-predictable process satisfying
$$
\E\left[\int_{\real_+ \times \real} |Z_{(t,x)}|^2 \lambda_t dt \nu(dx) \right]
< \infty
\quad \mbox{and}
\quad
 \E\left[ \left(\int_{\real_+ \times \real} Z_{(t,x)} \lambda_t dt \nu(dx)\right)^2\right]<\infty.
$$
 Then, for any $F\in L^2\big( \Omega^N,\mathcal F^N_\infty ,\P\big)$
  we have 
\begin{equation}
\label{eq:IBP}
\E\left[F \delta^N(Z \textbf{1}_{\{\theta \leq \lambda_\cdot \}})\right] = \E\left[ \int_{\real_+ \times \real} \lambda_t Z_{(t,x)} D_{(t,x)} F dt \nu(dx)\right].
\end{equation}
\end{theorem}
\noindent 
We conclude this section with a commutation property for the operators $D$ and $\delta^N$. 
 
\begin{lemma}
  Let $z:=(z(t,x))_{(t,x) \in \real_+\times\real} \in L^2(\real_+\times \real, dt\otimes \nu)$
  and consider
  $\mathcal{Z}:=(\mathcal{Z}_{(t,\theta)})_{(t,\theta) \in \real^2_+}$
  given in \eqref{eq:calZ}. 
 We have 
\begin{equation}
\label{eq:Ddelta}
D_{(t,x)} \delta^N(z \mathcal{Z}) = z(t,x) + \delta^N \big(z \widehat{\mathcal{Z}}^t\big),
\quad
t\geq 0, \ x\geq 0, 
\end{equation}
where 
\begin{equation}
\nonumber % \label{eq:hatZt}
\widehat{\mathcal{Z}}^t_{(r,\theta)}:= \textbf{1}_{\{r > t\}} \textbf{1}_{\{\lambda_r < \theta \leq \lambda_r \circ \varepsilon_{(t,0,1)}^+\}}, \quad r\in [t,+\infty), \
  \theta \geq 0.
\end{equation}
\end{lemma}
\begin{proof}
  The commutation relation \eqref{eq:Ddelta}
  can be derived in the ramework
  of the Malliavin calculus with respect to $N$. 
   In particular, according to \cite[Proposition 4.1.4]{Privault_LectureNotes}, for any $(t,\theta,x) \in \real_+^2\times \real$ we have 
   $$ D_{(t,\theta,x)} \delta^N(z \mathcal{Z}) = z(t,x) \mathcal{Z}_{(t,\theta)} + \delta^N(D_{(t,\theta,x)} (z \mathcal{Z})).
   $$
   By the definition \eqref{def:DH}
   of $D_{(t,x)}$ and the fact that $z$ is deterministic, we obtain 
$$ D_{(t,0,x)} \delta^N(z \mathcal{Z}) = z(t,x) + \delta^N(z D_{(t,0,x)} \mathcal{Z}).$$
In addition, as $\nu$ does not appear in the expression of $(H,\lambda)$ (see \eqref{eq:H}), we have 
\begin{align*}
D_{(t,0,x)} \mathcal{Z}_{(r,\theta)} &= (\mathcal{Z} \circ \varepsilon_{{(t,0,x)}^+\}})_{(r,\theta)} - \mathcal{Z}_{(r,\theta)} \\
&= \textbf{1}_{\{r > t\}} \big(\textbf{1}_{\{\theta \leq \lambda_r \circ \varepsilon_{(t,0,x)}^+\}}-\textbf{1}_{\{\theta \leq \lambda_r \}}\big)\\
&= \textbf{1}_{\{r > t\}} \textbf{1}_{\{\lambda_r < \theta \leq \lambda_r \circ \varepsilon_{(t,0 ,x)}^+\}}\\
&= \textbf{1}_{\{r > t\}} \textbf{1}_{\{\lambda_r < \theta \leq \lambda_r \circ \varepsilon_{(t,0 ,1)}^+\}}.
\end{align*}
where for the last equality, as remarked in the proof of Lemma \ref{lemma:lemma2},
 see also \eqref{eq:templambdashift}, $\lambda_r \circ \varepsilon_{(t,0 ,x)}^+$ does not depend on the value $x$ which thus can be taken equal to $x=1$. 
\end{proof}

\section{Main results}
\label{section:main}
\subsection{A general estimate}
In Theorem~\ref{th:main1} we present our main estimate
for functionals of the form 
\begin{align*}
  F & =\delta^N(z \mathcal{Z})
  \\
  & = \int_{\real_+^2\times \real} z(t,x) \mathcal{Z}_{(t,\theta)} \left(N(dt,d\theta,dx)-dt d \theta \nu(dx)\right)
  \\
  & = \int_{\real_+\times \real} z(t,x) \left(dH_t-\lambda_tdt  \nu(dx)\right),
  \end{align*} 
where 
$$ \mathcal{Z}_{(t,\theta)}= \textbf{1}_{\{\theta \leq \lambda_t\}}, \quad (t,\theta) \in \real_+^2,$$
and $z(t,x)$ is a deterministic square-integrable function. 
\begin{theorem}
\label{th:main1}
Let $z:=(z(t,x))_{(t,x)\in \real_+\times\real} \in L^2(\real_+\times \real,
dt\otimes \nu)$, $F:=\delta^N(z \mathcal{Z})$ and ${\cal N}_{\gamma^2} \sim \mathcal N(0,\gamma^2)$ with $\gamma^2 >0$.
It holds that:  
\begin{description}
\item{(i)} $\displaystyle
  d_W\big(F,{\cal N}_{\gamma^2}\big) \leq \E\left[\left|\gamma^2 - \int_{\real_+\times \real} |z(t,x)|^2 \lambda_t dt \nu(dx)\right| \right] $ 
\begin{equation}
\label{eq:rela1}
 + \E\left[\int_{\real_+\times \real} |z(t,x)| | D_{(t,x)} F |^2 \lambda_t dt \nu(dx) \right]. 
\end{equation}
\item{(ii)} If in addition, $z(t,x)$ satisfies
\begin{equation}
\label{eq:conditionZ}
z(t,x) \geq 0, \textrm{ for } dt \otimes \nu \textrm{ almost every } (t,x), 
\end{equation}
 then 
\begin{equation}
\label{eq:rela2}
d_W\big(F,{\cal N}_{\gamma^2}\big) \leq \E\left[\left|\gamma^2 - \int_{\real_+\times \real} |z(t,x)|^2 \lambda_t dt \nu(dx)\right| \right] + \E\big[F^3\big]. 
\end{equation}
\end{description} 
\end{theorem}
\begin{remark}
\label{rem:comment2}
We note that the term $\E\big[F^3\big]$ in \eqref{eq:rela2} is also the third
cumulant of the centered random variable $F$, 
see Section~4.8 of \cite{ChenGolsteinShao} on the ``third moment phenomenon''.
In particular, the vanishing of the first cumulant can lead to faster
convergence rates, see, e.g.
\cite{Privault,privaultthird}, 
\cite{nguyen3rd},
and Theorem~\ref{th:fast} below. 
\end{remark}

\subsection{Quantitative limit theorem for compound Hawkes processes}

Throughout this section we consider 
\begin{equation}
\nonumber % \label{eq:Xprocess}
S_t = \sum_{i=1}^{H_t} X_i, \quad t \geq 0, 
\end{equation}
defined through the system \eqref{eq:H},
 where $(X_i)_i$ is 
 a sequence of independent and identically distributed random variables
 with common distribution $\nu$,
 independent of the Hawkes process $H$ with intensity $\lambda$ given by \eqref{eq:lambda}. We will assume in addition that the kernel $\phi$ satisfies the following condition.
\begin{assumption}
\label{assumptionPhi}
The function $\phi : \real_+\to \real_+$ is such that 
$$ \|\phi\|_1=\int_0^\infty \phi(u) du < 1,
\quad \mbox{and} \quad 
\int_0^\infty u \phi(u) du < +\infty.
$$
\end{assumption}
\noindent 
Assumption~\ref{assumptionPhi} allows us to define 
\begin{equation}
\label{eq:Psi}
\psi:=\sum_{n\geq 1} \phi ^{(*n)},
\end{equation} 
where $\phi^{(*n)}$ is the $n$-th convolution of $\phi$ with itself, with 
$$  \int_0^\infty \psi (t) dt =\int_0^\infty \sum_{n\geq 1} \phi ^{(*n)}(t)dt
    =\sum_{n\geq 1}\int_0^\infty \phi^{(*n)}(t)dt
    =\sum_{n\geq 1} \|\phi\|_1^n
    =\frac{\|\phi\|_1}{1-\|\phi\|_1}.$$
\begin{theorem}[Quantitative limit theorem for the Hawkes process]
\label{th:main2}
 Assume that $\E[X_1^2] <+\infty$ and that Assumption~\ref{assumptionPhi} holds,
 and set 
 $$
  \gamma^2:= \mu \frac{\vartheta^2}{1-\|\phi\|_1}
  \quad \mbox{and}
  \quad
  F_T :=\frac{S_T- \E[X_1] \int_0^T \lambda_t dt}{\sqrt{T}},
 \qquad T>0. 
 $$ 
Then, there exists $C_{\mu,\phi,\vartheta} >0$ (depending only on $\mu, \|\phi\|_1, \vartheta$) such that 
\begin{equation}
\nonumber % \label{eq:rela3}
d_W(F_T,\mathcal{N}(0,\gamma^2)) \leq \frac{C_{\mu,\phi,\vartheta}}{\sqrt{T}},
\qquad T>0. 
\end{equation}
\end{theorem}
\begin{remark}
\label{rem:comment1}
As noted in its proof, this bound relies on the approach presented in \cite{torrisi,HHKR}. However,
 the bounds stated in those papers
 do not imply that $d_W\big(F_T,{\cal N}_{\gamma^2}\big)$ vanishes  as $T$ tends to $+\infty$ for any kernel $\phi$ satisfying Assumption~\ref{assumptionPhi}.
  Indeed, \cite[Theorem~3.1]{torrisi} (on which \cite[Relations (5.2) and (5.4)]{torrisi} rely on) entails 
$$
d_W(F_T, \mathcal{N}(0,{\sigma}^2)) \leq B_T, \ \!  \textrm{ with } \ \! 
B_T \geq
\sqrt{\frac{8}{\pi}} \|\phi\|_1
\frac{2-\|\phi\|_1}{1-\|\phi\|_1},
$$
 that does not converge to $0$. The situation is similar in \cite{HHKR},
 which contains the bound 
$$ d_W(F_T, \mathcal{N}(0,{\sigma}^2)) \leq \frac{C_{\mu,\phi,\vartheta}}{\sqrt{T}} + R_T,$$
 where $R_T$ is not shown to converge to $0$ as $T$ goes to $+\infty$, except in the particular cases of exponential and Erlang kernels,
 see  \cite[Theorems~3.10 and 3.12]{HHKR}  for a precise statement. The common point of these two approaches is a too stringent estimation in \eqref{eq:temp1} in the proof of Theorem~\ref{th:main1}. Here, pushing the computation one step forward and taking advantage of a deterministic integrand, we obtain a sharper result by making the term $R_T$ disappear from the bound. 
\end{remark}
\noindent
 In Theorem~\ref{th:fast} we provide a faster rate of convergence by considering a set of smoother test functions in the definition of the distance $d_{4,\infty}$ (see Notation~\ref{def:W} in Appendix~\ref{section:Stein}) under additional conditions on the moments of the random variables $X_i$, as in the following quantitative central limit theorem given of \cite{ChenGolsteinShao}. 

\begin{prop}
% \label{prop:Chen}
      (\cite[Corollary 4.4 in Section 4.8]{ChenGolsteinShao})
  Assume that 
  $$ \E[X_1] = \E\big[X_1^3\big] =0, \quad \E[X_1^2]=1, \quad
  \mbox{and} \quad \E[X_1^4] = \int_{-\infty}^\infty x^4 \nu(dx) <+\infty.
 $$
 Then, letting $W_n :=n^{-1/2} \sum_{i=1}^n X_i$, $n\geq 1$, we have 
$$ d_{4,\infty}(W_n ,\mathcal{N}(0,1)) \leq \frac{11+\E[X_1^4]}{24 n}.$$
\end{prop}
\noindent
The next result improves the rate of convergence $T^{-1/2}$ of
Theorem~\ref{th:main2} to $T^{-1}$ under additional assumptions on the random sequence $(X_i)_i$. 
\begin{theorem}
\label{th:fast}
Assume that Assumption~\ref{assumptionPhi} holds,
that $\mu \geq 0$, and
that
$$ \E[X_1] = \E\big[X_1^3\big] =0, \quad \E[X_1^4] = \int_{-\infty}^\infty x^4 \nu(dx) <+\infty,
$$
and let
$$\gamma^2:= \mu \frac{\vartheta^2}{1-\|\phi\|_1}
\quad
\mbox{and}
\quad
\vartheta^2:=\E[X_1^2].
$$
Then, there exists $C_{\mu,\phi,\vartheta} >0$ (depending only on $\mu, \|\phi\|_1, \vartheta$) such that 
\begin{equation}
\nonumber % \label{eq:rela3}
d_{4,\infty}({S_T}/\sqrt{T},\mathcal{N}(0,\gamma^2)) \leq \frac{C_{\mu,\phi,\vartheta}}{T},
\qquad T>0. 
\end{equation}
\end{theorem}

\noindent
 Finally, as \cite{HHKR}, we provide an alternative quantitative limit theorem by replacing the intensity process in the renormalisation with its {asymptotic} expectation. 
 This result extends the Wasserstein bound of \cite[Theorem~3.13]{HHKR}
 from simple Hawkes processes to compound Hawkes processes. 
\begin{theorem} 
\label{th:alternativ}
Assume that $\E[X_1^2] <+\infty$ and that Assumption~\ref{assumptionPhi} holds,
 and let 
 $$
 \varpi:=\mu \frac{\E[X_1]}{1-\|\phi\|_1}
\quad
\mbox{and}
 \quad 
 \Gamma_T:=\frac{S_T-\varpi T}{\sqrt T}, \quad T>0.
$$
 Then, 
 there exists $C_{\mu,\phi,\vartheta} >0$ (depending only on $\mu,\|\phi\|_1,\vartheta$)
 such that 
$$
 d_W (\Gamma_T,\mathcal N(0,\zeta^2) )
 \leq \frac{C_{\mu,\phi,\vartheta}}{\sqrt T}, \quad T>0,
 $$
where 
$$\zeta^2:=\mu\frac{\vartheta^2+\|\phi\|_1(\vartheta^2- ( \E[X_1] )^2)(\|\phi\|_1-2)}{(1-\|\phi\|_1)^3}. 
$$
\end{theorem}
\section{Proofs}
\label{section:proofs}

\subsection{Technical lemmata}

\begin{lemma}
\label{lemma:lemma1}
Let $z:=(z(t,x))_{(t,x) \in \real_+\times \real} \in L^2(\real_+\times \real, dt\otimes \nu)$
and consider $\mathcal{Z}$ defined in \eqref{eq:calZ}.
Then for any $t\geq 0$ and $\varphi : \real \to \real$ such that $\varphi\big(\delta^N(z \mathcal{Z})\big)$ belongs to $L^2\big(\Omega^N,\P\big)$, we have 
$\E_t\big[\varphi\big( \delta^N(z \mathcal{Z})\big)\delta^N\big(z \widehat{\mathcal{Z}}^t\big)\big]=0$. 
\end{lemma}
\begin{proof}
 Letting $F:=\delta^N(z \mathcal{Z})$, we have 
\begin{eqnarray*}
   \E_t\big[\varphi(F) \delta^N\big(z \widehat{\mathcal{Z}}^t\big) \big]
&=& \int_t^\infty \int_{\real_+ \times \real} \E_t\big[z(s,y) \textbf{1}_{\{\lambda_s < \theta \leq \lambda_s \circ \varepsilon_{(t,0,x)}^+\}} D_{(s,\theta,y)} \varphi(F) \big] d\theta \nu(dy) ds, 
\end{eqnarray*}
 with, by definition, 
\begin{align} 
  \nonumber
 &   \textbf{1}_{\{\lambda_s < \theta\}} D_{(s,\theta,y)} \varphi(F) 
  = \textbf{1}_{\{\lambda_s < \theta\}} \left[ \varphi\left( \left(\int_{\real_+^2 \times \real} z(t,x) \mathcal{Z}_{t,\rho} (N(dt,d\rho,dx)-dt d\rho \nu(dx))\right) \circ \epsilon_{(s,\theta,y)}^+\right) \right.
    \\
\label{eq:templemma1} 
 & \left.
 \quad - \varphi\left(\int_{\real_+^2 \times \real} z(t,x) \mathcal{Z}_{(t,\rho)} (N(dt,d\rho,dx)-dt d\rho \nu(dx))\right)\right]
\\
 \nonumber
& =  0,
\end{align} 
as $\textbf{1}_{\{\lambda_s < \theta\}} \mathcal{Z}_{(s,\theta)} = 0$.
More precisely, let 
$$ F_t:= \int_{[0,t]\times \real_+ \times \real} z(s,x) \mathcal{Z}_{s,\rho} (N(ds,d\rho,dx)-ds d\rho \nu(dx)),
\qquad
 t\geq 0,
$$
so that 
$$ F_\infty = \int_{\real_+^2 \times \real} z(s,x) \mathcal{Z}_{s,\rho} (N(ds,d\rho,dx)-ds d\rho \nu(dx)),
$$
 and $(F_t,\lambda_t)_{t\in \real}$ is solution to the SDE 
\begin{equation*}
\left\lbrace
\begin{array}{l}
  F_t = \displaystyle
  \int_{(0,t]\times \real_+\times \real} z(s,x) \textbf{1}_{\{\rho \leq \lambda_s\}} (N(ds,d\rho,dx)-ds d\rho \nu(dx)), \quad t \geq 0,

\medskip

\\
\lambda_t = \displaystyle \mu + \int_{(0,t)} \phi(t-u) d H_u,\quad t \geq 0 .
\end{array}
\right.
\end{equation*}
 Next, fix $(s,\theta,y) \in \real_+^2\times \real$.
 As $\textbf{1}_{\{\lambda_s < \theta\}} \mathcal{Z}_{(s,\theta)} = 0$, we have  
$$F_t \circ \varepsilon_{(s,\theta,y)}^+ = F_t, \quad \lambda_{t} \circ \varepsilon_{(s,\theta,y)}^+ = \lambda_{t}, \quad 0 \leq t < s,
$$
and 
$$\textbf{1}_{\{\lambda_s < \theta\}} F_s \circ \varepsilon_{(s,\theta,y)}^+ = \textbf{1}_{\{\lambda_s < \theta\}} F_s, \quad \textbf{1}_{\{\lambda_s < \theta\}}  \lambda_{s^+} \circ \varepsilon_{(s,\theta,y)}^+ = \textbf{1}_{\{\lambda_s < \theta\}}  \lambda_{s^+}.
$$ 
 As the process $\lambda_{s^+} \circ \varepsilon_{(s,\theta,y)}^+$ triggers the jumps of $F \circ \varepsilon_{(s,\theta,y)}^+$ and since it coincides with $\lambda$ on $[0,s]$, the pair $(F \circ \varepsilon_{(s,\theta,y)}^+, \lambda \circ \varepsilon_{(s,\theta,y)}^+)$ solves the same SDE as $(F,\lambda)$ and thus coincides with it by uniqueness of the solution\footnote{As $F$ is a counting process and $\lambda$ is deterministic between two jumps times of $F$, the uniqueness property is proved pathwise by considering the jump times of $F$ which are completely determined by the Poisson point process $N$ and the intensity process $\lambda$. We provide the details on the method of proof in the proof of Lemma~\ref{lemma:lemma2} below where a comparison principle is derived.}; which yields the last equality in~\eqref{eq:templemma1}. 
\end{proof}

\begin{lemma}
\label{lemma:lemma2}
 For all $(t,\eta , x)\in \real_+^2 \times \real$, it holds that:
$$ D_{(t,\eta ,x)} \lambda_s \geq 0, \quad ds\otimes \nu-\textrm{ a.e.}. $$
\end{lemma}

\begin{proof}
  The main proof idea relies on
  a comparison principle for the specific SDE involved,
  assessing that if intensity processes are ordered over
  the whole past (and not at some given time only),
  then this ordering between counting processes 
  and intensities is preserved at future times. 
We recall that $D_{(t,\eta ,x)} \lambda_s = \lambda_s \circ \epsilon_{(t,\eta ,x)}^+ - \lambda_s$. The conclusion follows from the definition of the intensity process $\lambda$ which is the unique solution to an SDE with respect to $N$ derived from \eqref{eq:H}:
$$ \lambda_s = \mu + \int_{(0,s)\times \real_+\times \real } \phi(s-u) \textbf{1}_{\{\theta \leq \lambda_u \}} N(du,d\theta,dx).$$
Note that we can write for $s\geq t$,
\begin{align}
\label{eq:templambda}
\lambda_s &= \mu + \int_{(0,t)\times \real_+\times \real } \phi(s-u) \textbf{1}_{\{\theta \leq \lambda_u \}} N(du,d\theta,dx) \nonumber\\
& \quad + \int_{(t,s)\times \real_+\times \real } \phi(s-u) \textbf{1}_{\{\theta \leq \lambda_u \}} N(du,d\theta,dx),
\end{align}
Similarly, for $s\geq t \geq 0$ we have 
\begin{align}
\label{eq:templambdashift}
\lambda_s \circ \epsilon_{(t,\eta ,x)}^+ &= \mu + \int_{(0,t)\times \real_+\times \real } \phi(s-u) \textbf{1}_{\{\theta \leq \lambda_u \}} N(du,d\theta,dx) \nonumber\\
& \quad + \phi(t-s) + \int_{(t,s)\times \real_+\times \real } \phi(s-u) \textbf{1}_{\{\theta \leq \lambda_u\circ \epsilon_{(t,\eta ,x)}^+ \}} N(du,d\theta,dx),
\end{align}
and $\lambda_s \circ \epsilon_{(t,\eta ,x)}^+ = \lambda_s $ for all $0 \leq s\leq t$,
$\eta \in \real_+$ and $x\in \real$. In addition, $\lambda_{t^+} \circ \epsilon_{(t,\eta ,x)}^+ = \lambda_{t^+} + \phi(0) \geq \lambda_{t^+}$. 
Note that $\lambda_s \circ \epsilon_{(t,\eta ,x)}^+$ does not depend on $x$ in Equation~\eqref{eq:templambdashift}. Let now 
$$\tau_1:=\inf\big\{s >t \ : \ \Delta_s H \circ \epsilon_{(t,\eta ,x)}^+ \neq 0\big\} \wedge \inf\big\{s >t\ : \ \Delta_s H \neq 0\big\},
$$ 
where we use the notation $\Delta_s H:=H_{s^+}-H_s$. So $\tau_1$ is the first jump of the Hawkes process $H$ or of the shifted Hawkes process $H \circ \epsilon_{(t,\eta ,x)}^+$ after $t$. Hence, from \eqref{eq:templambda}-\eqref{eq:templambdashift} we have 
$$ \lambda_s \circ \epsilon_{(t,\eta ,x)}^+ \geq \lambda_s, \quad s \in [0,\tau_1).$$
  Thus, at time $\tau_1$ (note that since $N$ is a Poisson point process we have $\tau_1<+\infty$, $\P$-a.s.), $N$ jumps at an atom $(\tau_1,\theta_1,x_1)$ which, by the previous ordering between $\lambda \circ \epsilon_{(t,\eta ,x)}^+$ and $\lambda$, imposes that $\tau_1$ is either a common jump time of $H \circ \epsilon_{(t,\eta ,x)}^+$ and $H$, or a jump time of $H \circ \epsilon_{(t,\eta ,x)}^+$ only,
 but it cannot be a jump time for $H$ and not for $H \circ \epsilon_{(t,\eta ,x)}^+$. In both situations (common jump or only a jump for the shifted Hawkes process) we have 
$$ \lambda_s \circ \epsilon_{(t,\eta ,x)}^+ \geq \lambda_s, \quad s \in [0,\tau_2),$$
where 
$$\tau_2:=\inf\big\{s >\tau_1 \ : \ \Delta_s H \circ \epsilon_{(t,\eta ,x)}^+ \neq 0\big\} \wedge \inf\big\{s >t \ : \ \Delta_s H \neq 0\big\}
$$ 
 is the next possible candidate jump time of $H$ and $H \circ \epsilon_{(t,\eta ,x)}^+$. Defining
 $$\tau_{n+1}:=\inf\big\{s >\tau_n \ : \ \Delta_s H \circ \epsilon_{(t,\eta ,x)}^+ \neq 0\big\} \wedge \inf\big\{s >t \ : \ \Delta_s H \neq 0\big\},
 \quad
 n \geq 0,
 $$ 
 and letting $n$ go to $+\infty$ concludes the proof.
\end{proof}

\noindent
 In what follows, for ease of notation we recall the notation
$\E_t[\ \! \cdot \ \! ] := \E[ \ \! \cdot \ \! \vert \mathcal F_t^N]$, $t\in \real_+$, and
 recall that by \eqref{eq:templambdashift} we have $\lambda \circ \varepsilon_{(t,0,x)}^+ =  \lambda \circ \varepsilon_{(t,0,1)}^+$ for any $(t,x)\in \real_+\times \real$.

\begin{lemma}
% \label{lemma:contrhat1}
For fixed any $0 \leq t < s \leq T$, let 
$\widehat{\lambda}_s^t:= \lambda_s \circ \varepsilon_{(t,0,1)}^+ - \lambda_s$,
and
\begin{equation}
\label{eq:hatH}
\widehat{H}_s^t := D_{(t,x)} H_s -1 = \int_{(t,s] \times \real_+ \times \real} \textbf{1}_{\{\lambda_r < \theta \leq \lambda_r \circ \varepsilon_{(t,0,1)}^+\}} N(dr,d\theta,dx). 
\end{equation}
Then $\big(\widehat{H}^t,\widehat{\lambda}^t\big)$ is a generalized Hawkes process (with a baseline intensity $\phi(\cdot-t)$) in the sense that $\widehat{H}^t$ is a counting process starting at time $t$ with stochastic intensity process $\widehat{\lambda}^t$. In addition, $\widehat{\lambda}^t$ satisfies 
\begin{equation}
\label{eq:hatlambda}
\widehat{\lambda}_s^{t} =  \phi(s-t) + \displaystyle{\int_{(t,s)} \phi(s-u) d\widehat{H}_u^{t}}, \quad s>t, \ \!  \widehat{\lambda}_t^{t}=0, 
\end{equation}
 and the following relations hold true: 
\begin{itemize}
\item[(i)] for any $p>0$ 
\begin{equation}
\label{eq:inequalityhatlambda}
\textrm{\rm esssup}_{t\in [0,T]} \left(\E_t\left[\int_t^T \widehat{\lambda}_s^t ds\right]\right)^p \leq (\|\phi\|_1 (1+\|\psi\|_1))^p, \quad \P-a.s..
\end{equation}
\item[(ii)] 
$\displaystyle 
\int_t^T \big(\widehat{\lambda}_s^t -\E_t\big[\widehat{\lambda}_s^t\big]\big) ds = \int_t^T \psi(T-s) \widehat{\mathcal{M}}_s^t ds$, $s\in (t,T]$, where $\widehat{\mathcal{M}}^t_s=\widehat{H}^t_s-\int_t^s \widehat{\lambda}^t_r dr$.
\item[(iii)] For any $p>0$, there exists $C>0$ depending on $p$ only and such that 
\begin{equation}
\label{eq:inequality2hatlambda}
\textrm{\rm esssup}_{t\in [0,T] } \E_t\left[\left|\int_t^T \widehat{\lambda}_s^t ds\right|^{2}\right] \leq C \|\psi\|_1^2 < +\infty.
\end{equation}
\end{itemize}
\end{lemma}

\begin{proof}
  Throughout this proof, $C>0$ denotes a positive constant that may change from line to line and is independent of $t$ and $T$. 
 We note that $\big(\widehat{H}^t,\widehat{\lambda}^t\big)$ is a generalized Hawkes process (with baseline intensity $0$) in the sense that $\widehat{H}^t$ is a counting process starting at time $t$ with stochastic intensity process $\lambda^t$. Indeed, the process $\widehat{\mathcal{M}}^t$ defined as 
\begin{align*}
  \widehat{\mathcal{M}}_{s}^t&:= \delta^N \big(
  \textbf{1}_{\{t < r \leq s\}} \textbf{1}_{\{\lambda_r < \theta \leq \lambda_r \circ \varepsilon_{(t,0,1)}^+\}} \big)
  \\
&=\int_{(t,s] \times \real_+ \times \real} \textbf{1}_{\{\lambda_r < \theta \leq \lambda_r \circ \varepsilon_{(t,0,1)}^+\}} \left(N(dr,d\theta,dx) -dr d\theta \nu(dx)\right)
  \\
&= \widehat{H}_s^t -\int_t^s \widehat{\lambda}_r^t dr, 
\end{align*}
is a $(\mathcal F_s^N)_{s\geq t}$-martingale by \eqref{eq:hatH}-\eqref{eq:hatlambda}.
 We refer to \cite[Proposition 2.19]{HHKR} for more details regarding the link between $\big(\widehat{H}^t,\widehat{\lambda}^t\big)$ and the Malliavin derivative of $(H,\lambda)$.
 \\
 \noindent $(i)$ 
A direct computation leads to (see \cite[Proof of Lemma 4.2]{HHKR}) 
\begin{equation}
\nonumber % \label{eq:esplambdahat}
\E_t\left[\int_t^T \widehat{\lambda}_s^t ds\right]  = \int_t^T \phi(s-t) ds + \int_t^T \int_u^T \psi(T-u) ds \phi(u-t) du, 
\end{equation}
and in turns to (see \cite[Lemma 4.2]{HHKR})
$$\E_t\left[\int_t^T \widehat{\lambda}_s^t ds\right] \leq \|\phi\|_1 (1+\|\psi\|_1), \quad \P-a.s.,$$
which gives \eqref{eq:inequalityhatlambda}.
\\
$(ii)$ 
 This result is similar to the one for a Hawkes process with constant baseline intensity $\mu$. However, to make this paper self-contained we present the proof below. 
Recall first that from \cite[Lemma 3]{Bacry_et_al_2013}, given a locally bounded map $h:\real_+\to \real$, the unique solution to equation
$$f(s)=h(s)+\int_0^s \phi (s-u) f(u)du,$$
is given by
\begin{equation}
\label{eq:Lemma3Bacry}
f(s)=h(s)+\int_0^s \psi(s-u)h(u)du.
\end{equation}
As $\widehat{\mathcal{M}}_{s}^t= \widehat{H}_s^t -\int_t^s \widehat{\lambda}_r^t dr$, we have  
$$\widehat{\lambda}_s^t=\phi(s-t)+ \int_t^s \phi(s-u)d\widehat{\mathcal{M}}_u^t+\int_t^s \phi(s-u)\widehat{\lambda}_u^tdu.$$
Taking the conditional expectation $\E_t[\ \! \cdot \ \! ]$, we get 
$$\mathbb E_t \big[\widehat{\lambda}_s^t\big]=\phi(s-t) +\int_t^s \phi(s-u)\mathbb E_t\big[\widehat{\lambda}_u^t\big]du,
$$
 which leads to 
$$\widehat{\lambda}_s^t-\mathbb E_t \big[\widehat{\lambda}_s^t\big]=\int_t^s \phi(s-u)d\widehat{\mathcal{M}}_u^t +\int_t^s \phi(s-u)\big(\widehat{\lambda}_u^t-\mathbb E_t \big[\widehat{\lambda}_u^t\big]\big)du.
$$
This expression is true for any $s>t$, and can be extended to any $s>0$ as follows: 
$$ \textbf{1}_{\{s>t\}}\big(\widehat{\lambda}_s^t-\mathbb E_t\big[\widehat{\lambda}_s^t\big]\big)=\textbf{1}_{\{s>t\}}\int_t^s \phi(s-u)d\widehat{\mathcal{M}}_u^t +\textbf{1}_{\{s>t\}}\int_0^s \phi(s-u)\textbf{1}_{\{u>t\}}\big(\widehat{\lambda}_u^t-\mathbb E \big[\widehat{\lambda}_u^t\big]\big)du.
$$
Letting $f_t(s):=\textbf{1}_{\{s>t\}}\big(\widehat{\lambda}_s^t-\mathbb E_t \big[\widehat{\lambda}_s^t\big]\big)$
being defined for any $s>0$, and vanishing if $0\leq s\leq t$,
and $h_t(s):=\textbf{1}_{\{s>t\}}\int_t^s \phi(s-u)d\widehat{\mathcal{M}}_u^t$,
 the above previous relation rewrites as 
$$f_t(s)=h_t(s)+\textbf{1}_{\{s>t\}}\int_0^s \phi(s-u)f_t(u)du.$$
As $\{s<t\}$ implies $\{u<t\}$ and so $f_t(u)=0$, the indicator function can be removed and thus 
$$f_t(s)=h_t(s)+\int_0^s \phi(s-u)f_t(u)du.$$
Applying \eqref{eq:Lemma3Bacry} we thus get $f_t(s)=h_t(s)+\int_0^s\psi(s-u)h_t(u)du$,
 which means
\begin{align*}
    \textbf{1}_{\{s>t\}}\big(\widehat{\lambda}_s^t-\mathbb E_t \big[\widehat{\lambda}_s^t\big]\big) &= \textbf{1}_{\{s>t\}}\int_t^s \phi(s-u)d\widehat{\mathcal{M}}_u^t + \int_0^s\psi(s-u) \textbf{1}_{\{u>t\}}\int_t^u \phi(u-v)d\widehat{\mathcal{M}}_v^t du.
\end{align*}
Next, using Fubini's theorem,
the fact that the stochastic integrals are defined pathwise,
 and the definition \eqref{eq:Psi} of $\psi$, we have
\begin{align*}
  &   \int_t^T \big(\widehat{\lambda}_s^t-\mathbb E_t \big[\widehat{\lambda}_s^t\big]\big) ds = \int_t^T \int_t^s \phi(s-u)d\widehat{\mathcal{M}}_u^t ds \\
  & \qquad \qquad \qquad \qquad \quad \quad 
  + \int_t^T \int_0^s\psi(s-u) \textbf{1}_{\{u>t\}}\int_t^u \phi(u-v)d\widehat{\mathcal{M}}_v^t du ds
  \\
  & \quad =\int_t^T \int_t^s \phi(s-u)d\widehat{\mathcal{M}}_u^t ds + \int_t^T \int_t^s\int_t^s\psi(s-u) \textbf{1}_{\{u>v>t\}} \phi(u-v)d\widehat{\mathcal{M}}_v^t du ds
  \\
  & \quad =\int_t^T \int_t^s \phi(s-u)d\widehat{\mathcal{M}}_u^t ds + \int_t^T \int_t^s\int_v^s\psi(s-u)  \phi(u-v)dud\widehat{\mathcal{M}}_v^t ds
  \\
  & \quad =\int_t^T \int_t^s \phi(s-u)d\widehat{\mathcal{M}}_u^t ds + \int_t^T \int_t^s\left(\int_0^{s-v}\psi(s-v-z)  \phi(z)dz \right) d\widehat{\mathcal{M}}_v^t ds
  \\
  & \quad =\int_t^T \int_t^s \phi(s-u)d\widehat{\mathcal{M}}_u^t ds + \int_t^T \int_t^s\left(\psi * \phi \right)(s-v) d\widehat{\mathcal{M}}_v^t ds
  \\
  & \quad =\int_t^T \int_t^s \phi(s-u)d\widehat{\mathcal{M}}_u^t ds + \int_t^T \int_t^s\left(\sum_{n\geq 1} \phi^{*n}* \phi \right)(s-v) d\widehat{\mathcal{M}}_v^t ds
  \\
  & \quad =\int_t^T \int_t^s \phi(s-u)d\widehat{\mathcal{M}}_u^t ds + \int_t^T \int_t^s\left(\sum_{n\geq 2} \phi^{*n} \right)(s-v) d\widehat{\mathcal{M}}_v^t ds
  \\
  & \quad =\int_t^T \int_t^s \phi(s-u)d\widehat{\mathcal{M}}_u^t ds + \int_t^T \int_t^s\left(\psi(s-v)-\phi(s-v)\right) d\widehat{\mathcal{M}}_v^t ds
  \\
  & \quad =\int_t^T \int_t^s\psi(s-v)d\widehat{\mathcal{M}}_v^t ds
  \\
  &\quad =\int_t^T \int_v^T \psi(s-v)dsd\widehat{\mathcal{M}}_v^t
    \\
    & \quad =\int_t^T \psi(T-v) \widehat{\mathcal{M}}_v^t dv, 
\end{align*}
where we applied again Fubini's theorem, 
and integration by parts over $[t,T]$. 
\\
$(iii)$ By the Burkholder-Davis-Gundy inequality, we have 
$$\textrm{\rm esssup}_{\{ (t,s) \ : \ 0\leq t\leq s\leq T \}} \E_t\big[\big|\widehat{\mathcal{M}}_s^t\big|^2\big] \leq C \textrm{\rm esssup}_{t\in [0,T]} \E_t\big[\widehat{H}_T^t\big] = C \textrm{\rm esssup}_{t\in [0,T]} \E_t\left[\int_t^T \widehat{\lambda}_s^t ds\right]$$
leading to 
\begin{equation}
\nonumber % \label{eq:hatmatuniformbounded}
\textrm{\rm esssup}_{\{ (t,s) \ : \ 0\leq t\leq s\leq T \}} \E_t\big[\big|\widehat{\mathcal{M}}_s^t\big|^2\big] <+\infty, \quad \P-a.s..
\end{equation}
This relation shows that $(iii)$ is a consequence of $(ii)$. Indeed, if $(ii)$ is satisfied, then 
Cauchy-Schwarz's inequality implies 
\begin{align*}
\E_t\left[\left(\int_t^T \big(\widehat{\lambda}_s^t -\E_t\big[\widehat{\lambda}_s^t\big]\big) ds\right)^2\right] &=\E_t\left[\left(\int_t^T \psi(T-s) \widehat{\mathcal{M}}_s^t ds\right)^2\right] \\
&= 2 \int_t^T \int_s^T \psi(T-s) \psi(T-u) \E_t\big[\widehat{\mathcal{M}}_s^t \widehat{\mathcal{M}}_u^t\big] du ds \\
&\leq C \|\psi\|_1^2.
\end{align*}
Combining this estimate with \eqref{eq:inequalityhatlambda} we get immediately that for any $p>0$, 
$$\textrm{\rm esssup}_{t\in [0,T]} \left( \E_t\left[\left|\int_t^T \widehat{\lambda}_s^t ds\right|^{2}\right] \right)^p \leq C <+\infty,$$
where we recall that $C>0$ does not depend on $t,T$,
which proves $(iii)$.
\end{proof}

\begin{lemma}
\label{lemma:contrhat2}
Letting 
$$
R_{t,T} := \int_{(t,T]\times \real_+\times \real} y \textbf{1}_{\{\lambda_r < \theta \leq \lambda_r \circ \varepsilon_{(t,0,1)}^+\}} \left(N(dr,d\theta,dy)-dr d\theta \nu(dy)\right),
  \quad
  0 \leq t\leq T,
$$
 we have  
$$ \textrm{\rm esssup}_{t\in [0,T] } \E_t\left[|R_{t,T}|^3\right] < \infty, \quad \P-a.s..$$
\end{lemma}
\begin{proof}
  Throughout this proof, $C>0$ denotes a positive constant that may change from line to line, and is independent of $t$ and $T$ considered below. 
 By the Burkholder-Davis-Gundy inequality, and the Jensen inequality we have 
\begin{align*}
  & \E_t\left[\left|R_{t,T}\right|^3\right] \leq C \E_t\left[\left| \int_{(t,T]\times \real_+\times \real} y^2 \textbf{1}_{\{\lambda_r < \theta \leq \lambda_r \circ \varepsilon_{(t,0,1)}^+\}} N(dr,d\theta,dy) \right|^{3/2}\right]
    \\
& \leq C \left(\E_t\left[\left| \int_{(t,T]\times \real_+\times \real} y^2 \textbf{1}_{\{\lambda_r < \theta \leq \lambda_r \circ \varepsilon_{(t,0,1)}^+\}} N(dr,d\theta,dy) \right|^{2}\right]\right)^{3/4}\\
& \leq C \left( \E_t\left[\left| \int_{(t,T]\times \real_+\times \real} y^2 \textbf{1}_{\{\lambda_r < \theta \leq \lambda_r \circ \varepsilon_{(t,0,1)}^+\}} (N(dr,d\theta,dy)-dr d\theta \nu(dy)) \right|^{2}\right]\right)^{3/4} \\
        & \quad + C \left(\int_{-\infty}^\infty y^2 \nu(dy)\right)^{3/2} \left( \E_t\left[\left| \int_t^T  \widehat \lambda_r^t dr \right|^{2}\right]\right)^{3/4}
        \\
&= C \left(\int_{-\infty}^\infty y^4 \nu(dy)\right)^{3/4} \left(\E_t\left[\int_t^T   \widehat \lambda_r^t dr\right]\right)^{3/4} + C \left(\int_{-\infty}^\infty y^2 \nu(dy)\right)^{3/2} \left(\E_t\left[\left| \int_t^T \widehat\lambda_r^t dr \right|^{2}\right] \right)^{3/4}, 
\end{align*}
 and the conclusion follows from \eqref{eq:inequalityhatlambda}-\eqref{eq:inequality2hatlambda}. 
\end{proof}

\begin{lemma}
\label{lemma:Rfrak}
For $T>0$, set
$$\mathfrak R_T := \frac{H_T-\int_0^T\E[\lambda_s]ds}{\sqrt T}-\frac{\mathcal M_T}{\sqrt T (1-\|\phi\|_1)}, \quad \mbox{and}
 \quad \mathcal M_T := H_T- \int_0^T \lambda_u du.$$
Under the assumptions of Theorem \ref{th:alternativ}, there exists a constant $C>0$ depending on $\mu,\|\phi\|_1$ such that
$$\E [\mathfrak R_T^2] \leq \frac{C}{T}.$$
\end{lemma}

\begin{proof}
Using \cite[Lemma 4]{Bacry_et_al_2013} we have that
\begin{align*}
    Y_T&:=\frac{1}{\sqrt T}\left(H_T-\E[H_T] \right)\\
    &=\frac{1}{\sqrt T}\left(H_T-\int_0^T \E[\lambda_t] dt \right)\\
    &=\frac{1}{\sqrt T}\left(\mathcal M_T+\int_0^T \psi(T-s)\mathcal M_s ds \right),
\end{align*}
hence
\begin{align*}
  Y_T-\frac{\mathcal M_T}{\sqrt T (1-\|\phi\|_1)}&=\frac{\mathcal M_T}{\sqrt T}\left( 1-\frac{1}{1-\|\phi\|_1}\right)
  + \frac{1}{\sqrt T} \int_0^T \psi(T-s) \mathcal M_s ds\\
    &=
    - \frac{1}{\sqrt T} \int_0^{+\infty} \psi(s) \mathcal M_T ds +
    \frac{1}{\sqrt T}
    \int_0^T \psi(T-s)
    \mathcal M_s ds,
\end{align*}
because $\int_0^{+\infty}\psi(s) ds =\|\phi\|_1/(1-\|\phi\|_1)$. Thus, if we set 
$$\mathfrak R_T = \frac{1}{\sqrt T} \left( \int_0^T \psi(s)(\mathcal M_{T-s}-\mathcal M_T)ds -\mathcal M_T\int_T^{+\infty} \psi(s)ds\right).$$
We have that
\begin{align*}
    \E[\mathfrak R_T^2]&\leq 2 \left(\frac{1}{T}\E[\mathcal M^2_T]\left( \int_T^{+\infty} \psi(s)ds\right)^2 + \frac{1}{T}\E \left[ \left( \int_0^{T}\psi(s)(\mathcal M_T-\mathcal M_{T-s})ds\right)^2\right] \right)\\
    &= 2 (A_1+A_2).
\end{align*}
Using the fact that the expected value of the square of a martingale is the expected value of its quadratic variation which in this case is the process jumps we have,
 using Assumption~\ref{assumptionPhi}, that 
\begin{align*}
    A_1 &= \frac{1}{T}\E\left[[\mathcal M]_T\right]\left( \int_T^{+\infty} \psi(s)ds\right)^2\\
    &=\frac{1}{T}\E [H_T ]\left( \int_T^{+\infty} \psi(s)ds\right)^2\\
    &=O\left( 1\right) \left( \int_T^{+\infty} \psi(s)ds\right)^2
    \\
    &\leq O\left( \frac{1}{T^2}\right) \left(\int_0^{+\infty} s\psi(s)ds\right)^2\\
    &=O\left(\frac{1}{T^2}\right)
\end{align*}
as in the proof of \cite[Lemma 5]{Bacry_et_al_2013}.
By expanding the square, the second term yields
\begin{align*}
    A_2&= \frac{1}{T}\E \left[ \left( \int_0^{T}\psi(s)(\mathcal M_T-\mathcal M_{T-s})ds\right)^2\right]\\
    &= \frac{2}{T}\E \left[  \int_0^{T}\int_0^s \psi(s)(\mathcal M_T-\mathcal M_{T-s})\psi(r)(\mathcal M_T-\mathcal M_{T-r}) dr ds \right]\\
    &= \frac{2}{T}  \int_0^{T}\int_0^s \psi(s)\psi(r)\E \left[(\mathcal M_T^2-\mathcal M_T\mathcal M_{T-s}-\mathcal M_T\mathcal M_{T-r}+\mathcal M_{T-r}\mathcal M_{T-s})\right] dr ds. \\
\end{align*}
Since for any $a\leq b$, $\E [\mathcal M_a \mathcal M_b]=\E [\mathcal M_a^2]=\E[H_a]$, we have that 
\begin{align*}
    A_2&=\frac{2}{T}  \int_0^{T}\int_0^s \psi(s)\psi(r)(\E  [ H_T]-\E[H_{T-s}]-\E[H_{T-r}]+\E[H_{T-s}] )dr ds\\
    &=\frac{2}{T}  \int_0^{T}\int_0^s \psi(s)\psi(r) \E  [ H_T- H_{T-r}] dr ds.
\end{align*}
In order to bound the integral, we use once again \cite[Lemma 4]{Bacry_et_al_2013} to obtain
\begin{align*}
    \E  [ H_T-H_{T-r}]&=r\mu +\left( r\int_0^{T-r}\psi(u)du  + \int_{T-r}^T \psi(u)(T-u)du\right)\mu\\
    &=r\mu +\left( r\int_0^{T-r}\psi(u)du  + \int_{0}^r \psi(T-u)udu\right)\mu\\
    &\leq r\mu +\left( r\|\phi\|_1  + \int_{0}^r \psi(T-u)rdu\right)\mu = O(r),
\end{align*}
and since $\psi$ is positive 
\begin{align*}
    A_2&\leq \frac{C}{T} \int_0^{T}\int_0^s \psi(s)\psi(r)r dr ds\\
    &\leq \frac{C}{T} \|\psi\|_1\int_0^{T} r\psi(r) dr\\
    &=O\left( \frac{1}{T}\right)
\end{align*}
which yields the desired result.
\end{proof}

\subsection{Proof of Theorem~\ref{th:main1}}
 According to Stein's method, see Appendix~\ref{section:Stein}, the Wasserstein distance between $F$ and ${\cal N}_{\gamma^2}$ can be bounded by 
\begin{equation}
\nonumber % \label{eq:Stein}
d_W\big(F,{\cal N}_{\gamma^2}\big) \leq \sup_{f \in \mathcal{F}_W} \left|\E[\gamma^2 f'(F) - F f(F)]\right|,
\end{equation}
 see \eqref{eq:Stein}, where 
$$\mathcal F_W:=\big\{f:\real \to \real,\ \!  \textrm{ twice differentiable with } \|f'\|_\infty \leq 1, \ \!  \|f''\|_\infty \leq 2 \big\}.$$
In addition, the right hand side is equal to $0$ if and only if $F\sim \mathcal{N}(0,\gamma^2)$.
\\
\noindent 
 $(i)$ We follow the beginning the Nourdin-Peccati's methodology (see \textit{e.g.} \cite{Peccati_2010})and apply the integration by parts formula to $\E[F f(F)]$ for $f$ in $\mathcal F_W$. More precisely, according to \cite[Proof of Theorem~3.4]{HHKR}, and using the integration by parts formula for the Hawkes process, see \eqref{eq:IBP}, we have 
	\begin{align*}
\E\left[f(F) F\right] & = \E\left[f(F) \delta^N(z \mathcal{Z})\right]\\
&= \E\left[\int_{\real_+\times \real} z(t,x) \lambda_t D_{(t,x)} f(F) dt \nu(dx)\right] \\
&= \E\left[\int_{\real_+\times} z(t,x) \lambda_t \left(f(F \circ \eps_{(t,0,x)}^+)-f(F)\right) dt\nu(dx)\right].
\end{align*}
 By Taylor expansion we have 
 $$ f(F \circ \eps_{(t,0,x)}^+)-f(F) = f'(F) D_{(t,x)} F + \frac12 f''\big(\widebar{F}^{t,x}\big) | D_{(t,x)} F |^2,
$$
 where $\widebar{F}^{t,x}$ denotes a random element between $F \circ \eps_{(t,0,x)}^+$ and $F$. 
 Hence we have 
\begin{align*}
\E\left[\gamma^2 f'(F) - f(F) F\right] 
  &= \E\left[f'(F) \left(\gamma^2 - \int_{\real_+\times \real} z(t,x) \lambda_t D_{(t,x)} F dt \nu(dx)\right) \right]
  \\
   & \quad - \frac{1}{2} \E\left[\int_{\real_+\times \real} z(t,x) \lambda_t f''\big(\widebar{F}^{t,x}\big) | D_{(t,x)} F |^2 dt \nu(dx)\right].
\end{align*}
At this stage, we provide a different treatment of the first time by expanding $D_{(t,x)} F$ according to \eqref{eq:Ddelta}. Thus
\begin{align*}
&\E\left[f'(F) \left(\gamma^2 - \int_{\real_+\times \real} z(t,x) \lambda_t D_{(t,x)} F dt \nu(dx)\right) \right] \\
&= \E\left[f'(F) \left(\gamma^2 - \int_{\real_+\times \real} |z(t,x)|^2 \lambda_t dt \nu(dx)\right) \right] - \E\left[f'(F) \int_{\real_+\times \real} z(t,x) \lambda_t \delta^N\big(z \widehat{\mathcal{Z}}^t\big) dt \nu(dx) \right],
\end{align*}
so that 
\begin{align}
\label{eq:temp1}
\E\left[\gamma^2 f'(F) - f(F) F\right] 
 & = \E\left[f'(F) \left(\gamma^2 - \int_{\real_+\times \real} |z(t,x)|^2 \lambda_t dt \nu(dx)\right) \right]
\\
\nonumber
& \quad - \frac{1}{2} \E\left[\int_{\real_+\times \real} z(t,x) \lambda_t f''\big(\widebar{F}^{t,x}\big) | D_{(t,x)} F |^2 dt \nu(dx)\right]
\\
\nonumber 
& \quad -  \int_{\real_+\times \real} z(t,x) \E\big[\lambda_t \E_t\big[f'(F)\delta^N\big(z \widehat{\mathcal{Z}}^t\big)\big] \big] dt \nu(dx).
\end{align}
We now compute the last term in \eqref{eq:temp1}. By Lemma \ref{lemma:lemma1}, we get 
$$ \E_t\big[f'(F)\delta^N\big(z \widehat{\mathcal{Z}}^t\big)\big]=0,
$$
and the estimate \eqref{eq:temp1} reads  
\begin{align}
\label{eq:temp3}
\E\left[\gamma^2 f'(F) - f(F) F\right] = & \E\left[f'(F) \left(\gamma^2 - \int_{\real_+\times \real} |z(t,x)|^2 \lambda_t dt \nu(dx)\right) \right]
\nonumber\\
&  - \frac{1}{2} \E\left[\int_{\real_+\times \real} z(t,x) \lambda_t f''\big(\widebar{F}^{t,x}\big) | D_{(t,x)} F |^2 dt \nu(dx)\right], 
\end{align}
which in turn implies 
\begin{align*}
  &\big|\E\big[\gamma^2 f'(F) - f(F) F\big] \big| \leq \left|\E\left[f'(F) \left(\gamma^2 - \int_{\real_+\times \real} |z(t,x)|^2 \lambda_t dt \nu(dx)\right) \right]\right|
  \\
   & \quad + \frac{1}{2} \left|\E\left[\int_{\real_+\times \real} z(t,x) \lambda_t f''\big(\widebar{F}^{t,x}\big) | D_{(t,x)} F |^2 dt \nu(dx)\right] \right|\\
  &\leq \|f'\|_\infty \E\left[\left|\gamma^2 - \int_{\real_+\times \real} |z(t,x)|^2 \lambda_t dt \nu(dx)\right| \right] + \frac{\|f''\|_\infty}{2} \E\left[\int_{\real_+\times \real}
        |z(t,x)|
     | D_{(t,x)} F |^2 \lambda_t dt \nu(dx)\right], 
\end{align*}
which leads to \eqref{eq:rela1} as $f$ belongs to $\mathcal F_W$.
\\
\noindent
$(ii)$ 
Using once again the integration by parts formula \eqref{eq:IBP} for the Hawkes process and the fact that (by an elementary algebraic computation) that
 $D (F^2) = |DF|^2 + 2 F DF$, we have 
\begin{align}
\label{eq:decompmoment3}
& \E\big[F^3\big] = \E\big[\delta^N(z \mathcal{Z}) F^2\big] \nonumber \\
&=\int_{\real_+\times \real} \E\left[z(t,x) \lambda_t D_{(t,x)} (F^2)\right] dt \nu(dx)\nonumber\\
&=\int_{\real_+\times \real} z(t,x) \E\left[ |D_{(t,x)} F|^2 \lambda_t \right] dt \nu(dx) + 2 \int_{\real_+\times \real} z(t,x) \E\left[\lambda_t \ \!  F D_{(t,x)} F\right] dt \nu(dx)\nonumber\\
&=:T_1 + 2 T_2.
\end{align}
Note that $T_1$ is exactly the second term in the right-hand side of \eqref{eq:rela1}, and thus the result follows if we show that $T_2 \geq 0$. To this end we compute this term. Using Relation~\eqref{eq:Ddelta} and integration by parts, we can expand this term as follows: 
\begin{align*}
T_2& =\int_{\real_+\times \real} z(t,x) \E\left[\lambda_t \ \!  F D_{(t,x)} F\right] dt \nu(dx)\\
&=\int_{\real_+\times \real} |z(t,x)|^2 \E\left[\lambda_t \ \!  F \right] dt \nu(dx)+\int_{\real_+\times \real} z(t,x)\E\big[ \lambda_t \ \!  F \delta^N \big(z \widehat{\mathcal{Z}}^t\big) \big] dt \nu(dx)\\
&=\int_{\real_+\times \real} |z(t,x)|^2 \E\left[ \lambda_t \ \! F \right] dt \nu(dx)+\int_{\real_+\times \real} z(t,x) \E\big[ \lambda_t \ \!  \E_t\big[F \delta^N \big(z \widehat{\mathcal{Z}}^t\big)\big] \big] dt \nu(dx)\\
&=\int_{\real_+\times \real} |z(t,x)|^2\E [ \lambda_t \ \!  F ] dt \nu(dx),
\end{align*}
where we applied Lemma \ref{lemma:lemma1}. By definition of $F$, we get that 
\begin{align*}
T_2&=\int_{\real_+\times \real} |z(t,x)|^2 \E [ \lambda_t \ \!  F ] dt \nu(dx) \\
&=\int_{\real_+\times \real} |z(t,x)|^2 \E [ \lambda_t \ \!  \delta^N(z \mathcal{Z}) ] dt \nu(dx) \\
&=\int_{\real_+\times \real}  |z(t,x)|^2 \E \left[ \int_{[0,t)\times \real} z(s,y) \lambda_s D_{(s,y)}
      \lambda_t  ds \nu(dy)\right] dt \nu(dx) \\
&=\int_{\real_+\times \real} \int_{[0,t)\times \real} z(s,y) |z(t,x)|^2 \E\left[ \lambda_s D_{(s,y)} \lambda_t \right] ds \nu(dy) dt \nu(dx) \\
&\geq 0, \quad \P-a.s., 
\end{align*}
where for the last inequality we used Lemma \ref{lemma:lemma2}
and the Assumption~\eqref{eq:conditionZ} on $z(t,x)$. 
To summarize, we have shown that 
\begin{equation}
\label{eq:ineqmoment3}
\E\big[F^3\big] \geq \E\left[\int_{\real_+\times \real} |z(t,x)| 
  | D_{(t,x)} F |^2 \lambda_t dt \nu(dx) \right].
\end{equation}

\subsection{Proof of Theorem~\ref{th:main2}}

Let $T>0$, and note that we have $F_T=\delta^N(z \mathcal{Z})$ with $z(t,x):= \textbf{1}_{\{t\in [0,T] \}} x / \sqrt{T}$, $(t,x)\in \real_+\times \real$.
By Relation~\eqref{eq:rela1} in Theorem~\ref{th:main1} we have 
\begin{equation}
\label{eq:tempproofmain2}
d_W\big(F,{\cal N}_{\gamma^2}\big) \leq \E\left[\left|\gamma^2 - \frac{\vartheta^2}{T} \int_0^T \lambda_t dt \right| \right] + \frac{1}{\sqrt{T}} \E\left[ \int_\real \int_0^T |x| | D_{(t,x)} F_T |^2 \lambda_t dt \nu(dx) \right]. 
\end{equation}
On the other hand, by \cite[Estimates on Term $A_2$-Proof of Theorem~3.10]{HHKR} we have 
$$ \frac{1}{\sqrt{T}} \E\left[ \int_\real \int_0^T |x|
  | D_{(t,x)} F_T |^2 \lambda_t dt \nu(dx) \right] =O(T^{-1/2}),
$$
for any $\phi$ satisfying Assumption~\ref{assumptionPhi}.
Regarding the first term in the right hand-side of \eqref{eq:tempproofmain2}, we have 
\begin{align*}
\E\left[\left|\gamma^2 - \frac{\vartheta^2}{T} \int_0^T \lambda_t dt \right| \right] 
&\leq \left|\gamma^2 - \frac{\vartheta^2}{T} \int_0^T \E[\lambda_t] dt \right| + \frac{\vartheta^2}{T}  \E\left[\left|\int_0^T (\lambda_t-\E[\lambda_t]) dt \right| \right] \\
&=O(T^{-1/2}),
\end{align*}
since 
$$\left|\gamma^2 - \frac{\vartheta^2}{T} \int_0^T \E[\lambda_t] dt \right| =O(T^{-1}),
\quad \frac{\vartheta^2}{T}  \E\left[\left|\int_0^T (\lambda_t-\E[\lambda_t]) dt \right| \right] =O(T^{-1/2}),
$$
 by \cite[Lemma 4.1]{HHKR} and \cite[Estimate on Term $A_{1,2}$ - Proof of Theorem~3.10]{HHKR}.
 \hfill $\Box$
 \\
 \noindent
 As $z(t,x)$ is non-negative, instead of \eqref{eq:rela1}
  we could also have applied \eqref{eq:rela2} which involves the quantity $\E\big[F_T^3\big]$, with 
  $$ \E\big[F_T^3\big] \geq \frac{1}{\sqrt{T}} \E\left[\int_0^T \int_\real |x| | D_{(t,x)} F_T |^2 \lambda_t dt \nu(dx) \right],
  $$
  by \eqref{eq:ineqmoment3},
  which recovers the convergence with decay rate $T^{-1/2}$ using the decomposition~\eqref{eq:decompmoment3}, in which the first term coincides with the bound in \eqref{eq:rela1} and the second term is proved to decay as $T^{-1/2}$ using \cite[Lemma 4.2]{HHKR}.

\subsection{Proof of Theorem~\ref{th:fast}}Let $T>0$. As $m=\int_{-\infty}^\infty x \nu(dx) = \E[X_1] =0$, $S_T = \delta^N(z \mathcal{Z})$ with $z(t,x)= x \textbf{1}_{\{t \in [0,T] \}} / \sqrt{T}$, 
 using \eqref{eq:Stein} in Appendix~\ref{section:Stein} it holds that : 
 $$ \|S_T-\mathcal{N}(0,\gamma^2)\|_{4,\infty} \leq \sup_{f\in \mathcal F_W^4} \big|\E\big[\gamma^2 f'(S_T)-S_Tf(S_T)\big]\big|,
 $$
 with $\gamma^2=\vartheta ^2 \mu/(1-\|\phi\|_1)$, where $\vartheta ^2=\E [X_1^2]$.
Following the lines of the proof of Theorem~\ref{th:main1} and using a Taylor expansion of order $3$ for $D_{(t,x)} f(S_T)$, we have 
$$ f\big(S_T \circ \eps_{(t,0,x)}^+\big) - f(S_T) = f'(S_T) D_{(t,x)} S_T + \frac12 f''(S_T) |D_{(t,x)} S_T|^2 + \frac16 f^{(3)}\big(\widebar{S}^{t,x}\big) (D_{(t,x)} S_T)^3,
$$
where $\widebar{S}^{t,x}$ denotes a random element between $S_T$ and $S_T \circ \eps_{(t,0,x)}^+$. Relation~\eqref{eq:temp3} then becomes
\begin{eqnarray*}
  \lefteqn{ 
    \! \! \! \! \! \! \! \! \! \! \! \! \! \! \!
    \! \! \!  \! \! \! 
    \E\left[\gamma^2 f'(S_T) - f(S_T) S_T\right] 
= \E\left[f'(S_T) \left(\gamma^2 - \int_{\real_+\times \real} z(t,x) D_{(t,x)} S_T \lambda_t dt \nu(dx)\right) \right]
  }
  \\
& & - \frac{1}{2} \E\left[\int_{\real_+\times \real} z(t,x) \lambda_t f''(S_T) | D_{(t,x)} S_T |^2 dt \nu(dx)\right] \\
& & - \frac{1}{6} \E\left[\int_{\real_+\times \real} z(t,x) \lambda_t f^{(3)}\big(\widebar{S}^{t,x}\big) (D_{(t,x)} S_T )^3 dt \nu(dx)\right]\\
&= & \E\left[f'(S_T) \left(\gamma^2 - \int_{\real_+\times \real} |z(t,x)|^2 \lambda_t dt \nu(dx)\right) \right]\\
& & - \frac{1}{2} \E\left[\int_{\real_+\times \real} z(t,x) \lambda_t f''(S_T) | D_{(t,x)} S_T |^2 dt \nu(dx)\right] \\
& & - \frac{1}{6} \E\left[\int_{\real_+\times \real} z(t,x) \lambda_t f^{(3)}\big(\widebar{S}^{t,x}\big) (D_{(t,x)} S_T )^3 dt \nu(dx)\right],
\end{eqnarray*}
where we used Relation~\eqref{eq:Ddelta} for the Malliavin derivative of $S_T$, i.e. 
\begin{equation}
\label{fjkldsf}
  D_{(t,x)} S_T = z(t,x) + \delta^N \big(z\widehat{\mathcal{Z}}^t\big), 
\end{equation} 
and Lemma \ref{lemma:lemma1}. It is important to notice that the quantity $\delta^N \big(z \widehat{\mathcal{Z}}^t\big)$ is independent of $x$ as $\lambda_r \circ \varepsilon_{(t,0,x)}^+$ is (see \ref{eq:templambdashift}), and thus we have 
\begin{align*}
  \delta^N \big(z\widehat{\mathcal{Z}}^t\big) &=\frac{1}{\sqrt{T}}\int_{(t,T]\times \real_+\times \real} y \textbf{1}_{\{\lambda_r < \theta \leq \lambda_r \circ \varepsilon_{(t,0,x)}^+\}} (N(dr,d\theta,dy)-dr d\theta \nu(dy) )
  \\
&=\frac{1}{\sqrt{T}}\int_{(t,T]\times \real_+\times \real} y \textbf{1}_{\{\lambda_r < \theta \leq \lambda_r \circ \varepsilon_{(t,0,1)}^+\}} (N(dr,d\theta,dy)-dr d\theta \nu(dy)).
\end{align*}
Using the definition of $z(t,x)$, we get 
\begin{eqnarray*}
  \lefteqn{
    \E\left[\gamma^2 f'(S_T) - f(S_T) S_T\right] 
    = \vartheta^2 \E\left[f'(S_T) \left(\frac{\mu}{1-\|\phi\|_1} - \frac{1}{T} \int_0^T \lambda_t dt \right) \right] }
\\
& & - \frac{1}{2 \sqrt{T}} \E\left[\int_0^T \int_{-\infty}^\infty x \lambda_t f''(S_T) | D_{(t,x)} S_T |^2 dt \nu(dx)\right] \\
& & - \frac{1}{6 \sqrt{T}} \E\left[\int_0^T \int_{-\infty}^\infty x \lambda_t f^{(3)}\big(\widebar{S}^{t,x}\big) (D_{(t,x)} S_T )^3 dt \nu(dx)\right]\\
&= &  \vartheta^2 \E\left[f'(S_T) \left(\frac{\mu}{1-\|\phi\|_1} - \frac{1}{T} \int_0^T \lambda_t dt \right) \right] 
 - \frac{1}{2 T^{3/2}} \int_{-\infty}^\infty x^3 \nu(dx) \E\left[\int_0^T \lambda_t f''(S_T) dt \right] \\
& & - \frac{\int_{-\infty}^\infty x \nu(dx)}{2 \sqrt{T}} \E\left[\int_0^T \lambda_t f''(S_T) \big| \delta^N \big(z\widehat{\mathcal{Z}}^t\big) \big|^2 dt \right]
 - \frac{\vartheta^2}{T} \E\left[\int_0^T \lambda_t \E_t\big[f''(S_T) \delta^N \big(z\widehat{\mathcal{Z}}^t\big)\big] dt \right] \\
& & - \frac{1}{6 \sqrt{T}} \E\left[\int_0^T \int_{-\infty}^\infty x \lambda_t f^{(3)}\big(\widebar{S}^{t,x}\big) (D_{(t,x)} S_T)^3 \nu(dx) dt \right]\\
&= &  \vartheta^2 \E\left[f'(S_T) \left(\frac{\mu}{1-\|\phi\|_1} - \frac{1}{T} \int_0^T \lambda_t dt \right) \right] \\
& & - \frac{1}{6 \sqrt{T}} \E\left[\int_0^T \int_{-\infty}^\infty x \lambda_t f^{(3)}\big(\widebar{S}^{t,x}\big) (D_{(t,x)} S_T)^3 \nu(dx) dt\right],
\end{eqnarray*}
where we used once again Lemma \ref{lemma:lemma1} and the fact that
$\int_{-\infty}^\infty x \nu(dx) = \int_{-\infty}^\infty x^3 \nu(dx)=0$. 
Hence, we have 
\begin{eqnarray}
\nonumber % \label{eq:fastestim1}
\lefteqn{
  \! \!   
  \left|\E\left[\gamma^2 f'(S_T) - f(S_T) S_T\right]\right| 
  \leq \vartheta^2 \|f'\|_\infty \left|\frac{\mu}{1-\|\phi\|_1} - \frac{1}{T} \int_0^T \E[\lambda_t] dt \right| }
\nonumber\\
& & +\frac{1}{T} \left| \E\left[f'(S_T) \int_0^T (\lambda_t-\E[\lambda_t]) dt \right]\right|  + \frac{2 \|f^{(3)}\|_\infty \int_{-\infty}^\infty x^4 \nu(dx) }{3T^2} \E\left[\int_0^T \lambda_t dt \right] \nonumber \\
& & + \frac{2 \|f^{(3)}\|_\infty \int_{-\infty}^\infty |x| \nu(dx) }{3 \sqrt{T}} \E\left[\int_0^T \lambda_t \big|\delta^N \big(z\widehat{\mathcal{Z}}^t\big)\big|^3 dt \right] \nonumber\\
&= &  \vartheta^2 \|f'\|_\infty \left|\frac{\mu}{1-\|\phi\|_1} - \frac{1}{T} \int_0^T \E[\lambda_t] dt \right| +\frac{1}{T} \left| \E\left[f'(S_T) \int_0^T (\lambda_t-\E[\lambda_t]) dt \right]\right| \nonumber\\
& & + \frac{2 \|f^{(3)}\|_\infty \int_{-\infty}^\infty x^4 \nu(dx) }{3T^2} \E\left[\int_0^T \lambda_t dt \right] + \frac{2 \|f^{(3)}\|_\infty \int_{-\infty}^\infty |x| \nu(dx) }{3 T^2} \E\left[\int_0^T |R_{t,T}|^3 \lambda_t dt \right] \nonumber\\
&=: &  A_1 + A_2 + A_3 + A_4,
\end{eqnarray} 
where we set 
$$R_{t,T}:=\sqrt{T} \delta^N \big(z\widehat{\mathcal{Z}}^t\big) = \int_{(t,T]\times \real_+\times \real} y \textbf{1}_{\{\lambda_r < \theta \leq \lambda_r \circ \varepsilon_{(t,0,1)}^+\}} \left(N(dr,d\theta,dy)-dr d\theta \nu(dy)\right).$$
  We now treat the above three terms separately.
  
\medskip

\noindent
First, note that by \cite[Lemma 4.1]{HHKR} we have $A_1 = O(T^{-1})$. In addition, as $\E\left[\int_0^T  \lambda_t dt \right] =O(T)$ and since $\E_t\left[|R_{t,T}|^3\right]$ is bounded uniformly in $t,T$ by Lemma \ref{lemma:contrhat2}, we have $A_3 + A_4= O(T^{-1})$. It remains to deal with the term $A_2$. Using \cite[Relation (14)]{Bacry_et_al_2013} and \cite[Proof of Theorem~3.10, Term $A_{1,2}$]{HHKR}, we get that 
$$ \int_0^T (\lambda_s - \E[\lambda_s]) ds = \int_0^T \psi(T-s) \mathcal M_s ds,$$
where $\mathcal M_s:= H_s- \int_0^s \lambda_u du$. Hence, we have 
\begin{eqnarray*}
  \lefteqn{ 
    \! \! \! \! \! \! \! \! \! 
    \E\left[f'(S_T) \int_0^T (\lambda_t-\E[\lambda_t]) dt \right]
 = \E\left[f'(S_T) \int_0^T \psi(T-s) \mathcal M_s ds \right]
  }
  \\
&=&\int_0^T \psi(T-s) \E\left[f'(S_T)  \mathcal M_s \right] ds \\
&=& \int_0^T \psi(T-s) \E\left[f'(S_T)  \delta^N(\mathcal{Z} \textrm{1}_{\{ \cdot \leq s \}}) \right] ds\\
&=& \int_0^T \psi(T-s) \int_0^{s} \int_{\real} \E\left[\lambda_u D_{(u,x)} f'(S_T)  \right] \nu(dx) du ds\\
&=& \int_0^T \psi(T-s) \int_0^{s} \int_{\real} \E\left[\lambda_u f''(S_T) D_{(u,x)} S_T \right] \nu(dx) du ds\\
&&+ \frac12 \int_0^T \psi(T-s) \int_0^{s} \int_{\real} \E\left[\lambda_u f^{(3)}\big(\widebar{S}^{u,x}\big) |D_{(u,x)} S_T|^2 \right] \nu(dx) du ds\\
  &=& \frac{1}{\sqrt{T}} \int_{\real} x  \nu(dx) \int_0^T \psi(T-s) \int_0^{s} \E\left[\lambda_u f''(S_T) \right] du ds
  \\
&&+ \int_0^T \psi(T-s) \int_0^{s} \int_{\real} \E\big[\lambda_u \E_u\big[f''(S_T) \delta^N\big(z \widehat{\mathcal{Z}}^u\big)\big] \big] \nu(dx) du ds\\
&&+ \frac12 \int_0^T \psi(T-s) \int_0^{s} \int_{\real} \E\left[\lambda_u f^{(3)}\big(\widebar{S}^{u,x}\big) |D_{(u,x)} S_T|^2 \right] \nu(dx) du ds\\
&=&\frac12 \int_0^T \psi(T-s) \int_0^{s} \int_{\real} \E\left[\lambda_u f^{(3)}\big(\widebar{S}^{u,x}\big) |D_{(u,x)} S_T|^2 \right] \nu(dx) du ds,
\end{eqnarray*}
as $\int_{\real} x  \nu(dx) = \E[X_1] =0$ and $\E_u\big[f''(S_T)\delta^N\big(z \widehat{\mathcal{Z}}^u\big)\big]=0$ by Lemma \ref{lemma:lemma1}, where $\widebar{S}^{u,x}$ denotes a random element between $S_T$ and ${S_T} \circ \varepsilon_{(u,0,x)}^+$. Hence, by \eqref{fjkldsf} we have 
\begin{align*}
  & \E\left[f'(S_T) \int_0^T (\lambda_t-\E[\lambda_t]) dt \right]
  \\
&=\frac12 \int_0^T \psi(T-s) \int_0^{s} \int_{\real} \E\big[\lambda_u f^{(3)}\big(\widebar{S}^{u,x}\big) |D_{(u,x)} S_T|^2 \big] \nu(dx) du ds \\
  &\leq
  \frac12 \|f^{(3)}\|_\infty
  \int_0^T \psi(T-s) \int_0^{s} \int_{\real} \E\big[\lambda_u |D_{(u,x)} S_T|^2 \big] \nu(dx) du ds \\
  &\leq
\frac{1}{T}
  \|f^{(3)}\|_\infty \|\psi\|_1 \int_{\real} x^2 \nu(dx)  \int_0^T  \E\left[\lambda_u \right] du \\
  & \quad + \|f^{(3)}\|_\infty \int_0^T \psi(T-s) \int_0^T \E\big[\lambda_u \E_u\big[\big|\delta^N\big(z\widehat{\mathcal{Z}}^u\big)\big|^2\big] \big]du ds.
\end{align*}
 Next, by the It\^o isometry, see \textit{e.g.}, \cite[Proposition 6.5.4]{Privault_LectureNotes}
  and
  (\ref{eq:inequalityhatlambda}),
 for some $C>0$ it holds that 
\begin{align*}
   \E_u\big[\big|\delta^N\big(\widehat{\mathcal{Z}}^u\big)\big|^2\big] 
   &= \frac{1}{T} \int_{\real} x^2 \nu(dx)
    \int_u^{T} \int_{0}^{+\infty} \E_u\big[\big|\textbf{1}_{\{\lambda_r < \theta \leq \lambda_r \circ \varepsilon_{(t,0,1)}^+\}}\big|^2\big]  d\theta ds \\
    &= \frac{1}{T} \int_{\real} x^2 \nu(dx)
     \int_u^{T} \E_u\big[ \widehat{\lambda}_s^u \big] ds \\
& \leq \frac{C}{T}, \qquad u\in [0,T]. 
\end{align*}
Thus, for some $C>0$ we have 
\begin{align*}
  \textrm{\rm esssup}_{T>0} \left|\E\left[f'(S_T) \int_0^T (\lambda_t-\E[\lambda_t]) dt \right]\right|&\leq C \|f^{(3)}\|_\infty \int_{\real} x^2 \nu(dx) \|\psi\|_1 \textrm{\rm esssup}_{T>0}
   \int_0^T  \frac{\E [\lambda_u ]}{T} du \\
   & \leq C,
\end{align*}  
as $\lim_{T\to+\infty} T^{-1} \int_0^T  \E\left[\lambda_u \right] du
= \mu / ( 1-\|\phi\|_1 )$,
 which shows that $A_2=O(T^{-1})$ and concludes the proof.
 
\subsection{Proof of Theorem \ref{th:alternativ}} 
 
We start by proving the upper bound on the Wasserstein distance between
the distribution of
$$
V_T:=\frac{S_T-m\int_0^T \E[\lambda_s]ds}{\sqrt T}
$$
and $\mathcal N(0,\zeta^2)$. For this, we consider the normalized martingale
$$F_t:=\frac{S_t-m\int_0^t \lambda_sds}{\sqrt t}
$$
associated to the compound process $S$, we write
\begin{align*}
    F_T&=\frac{1}{\sqrt T} \left( S_T-m\int_0^T\lambda_s ds\right)\\
    &=\frac{1}{\sqrt T} \left( S_T-mH_T\right)+m\frac{\mathcal M_T}{\sqrt T},
\end{align*}
where $\mathcal M_T := H_T- \int_0^T \lambda_u du$. Similarly, we have 
\begin{align*}
    V_T&=\frac{1}{\sqrt T}\left( S_T-m\int_0^T \E[\lambda_s]ds\right)\\
    &=\frac{1}{\sqrt T} \left( S_T-mH_T\right)+mY_T,
\end{align*}
with $Y_T:=  ( H_T-\int_0^T\E[\lambda_s]ds ) / \sqrt{T}$. 
Thus, 
\begin{align*}
    (1-\|\phi\|_1)V_T-F_T=-\|\phi\|_1 \frac{S_T-mH_T}{\sqrt T}+m(1-\|\phi\|_1)\mathfrak R_T,
\end{align*}
where we let  
$$\mathfrak R_T := Y_T-\frac{\mathcal M_T}{\sqrt T (1-\|\phi\|_1)}.$$
Next, we note that
\begin{align*}
 & \delta ^N\left(((x-m)\mathds 1_{t\leq T})_{(t,x)\in (\mathbb R_+\times \mathbb R)} \mathcal Z\right) = \int_0^T\int_{\mathbb R_+\times \mathbb R}(x-m)\textbf{1}_{\{\theta \leq \lambda_t \}}\left( N(dt,d\theta,dx)-dtd\theta \nu(dx)\right)
  \\
  & = \int_0^T\int_{\mathbb R_+\times \mathbb R}(x-m)\textbf{1}_{\{\theta \leq \lambda_t \}}N(dt,d\theta,dx)-\int_0^T\int_{\mathbb R_+\times \mathbb R}x\textbf{1}_{\{\theta \leq \lambda_t \}} dtd\theta \nu(dx)\\
    & \quad +m\int_0^T\int_{\mathbb R_+\times \mathbb R}\textbf{1}_{\{\theta \leq \lambda_t \}} dtd\theta \nu(dx)\\
   & = S_T-mH_T.
\end{align*}
Hence, using the fact that $F$ is also written as a divergence, we have 
\begin{align*}
    (1-\|\phi\|_1)V_T&=\delta^N\big(( z(t,x))_{(t,x)\in (\mathbb R_+\times \mathbb R)} \mathcal Z\big)+m(1-\|\phi\|_1)\mathfrak R_T,
\end{align*}
where
$$
z(t,x):=\frac{x+(m-x)\|\phi\|_1}{\sqrt T}\mathds 1_{[0,T]}(t). 
$$
We now proceed in the same manner as the proof of Theorem \ref{th:main2}, only by replacing $x$ by $(1-\|\phi\|_1)x+\|\phi\|_1m$. This proves that 
$\delta ^N\big (( z(t,x))_{(t,x)\in (\mathbb R_+\times \mathbb R)} \mathcal Z\big)$ converges to a centered Gaussian random variable of variance 
\begin{align*}
&\frac{\mu}{1-\|\phi\|_1}\int_{\mathbb R}((1-\|\phi\|_1)x+\|\phi\|_1m)^2\nu(dx)\\
&=\frac{\mu}{1-\|\phi\|_1}\int_{\mathbb R}((1-\|\phi\|_1)^2x^2+2(1-\|\phi\|_1) \|\phi\|_1xm+\|\phi\|_1^2m^2)\nu(dx)\\
&=\frac{\mu}{1-\|\phi\|_1}\left((1-2\|\phi\|_1+\|\phi\|_1^2)\vartheta^2 +2(1-\|\phi\|_1) \|\phi\|_1m^2+\|\phi\|_1^2m^2\right)\\
&=\frac{\mu}{1-\|\phi\|_1}\left( \vartheta^2+\|\phi\|_1(\vartheta^2-m^2) (\|\phi\|_1-2)\right)
\end{align*}
where the Wasserstein distance between the two variables is bounded by $O(T^{-1/2})$.
By proceeding as in the proof of Theorem 3.13 in \cite{HHKR} it is enough to show that $\E[\mathfrak R_T^2]=O(T^{-1})$, which is done in Lemma \ref{lemma:Rfrak},
to obtain 
$$d_W\left(V_T,\mathcal N(0,\zeta^2)\right)\leq \frac{C}{\sqrt T}.$$
Finally, by \cite[Lemma 5]{Bacry_et_al_2013}, we have 
$$ 
    V_T-\Gamma_T =\frac{m}{\sqrt T} \left( \int_0^T \E[\lambda_t]dt -\frac{\mu}{1-\|\phi\|_1}\right)
  =O(T^{-1/2}),
 $$ 
 which yields the desired result.
  
\appendix

\section{Elements of Stein's method}
\label{section:Stein}

In this section we describe elements
of Stein's method, which was has been introduced by C.M. Stein in \cite{Stein_method}, 
that are relevant to our analysis and to 
the derivation of bounds of the form \eqref{eq:Stein}.
We let $h^{(i)}$, $i\geq 1$, denotes the $i$th derivative of $h$. 
\noindent 
\begin{definition}
\label{def:W}
In what follows $\mathcal H$ denotes one of the following Hilbert spaces : 
\begin{enumerate}
\item $ \mathcal H_{W}:=\big\{h:\real \to \real \textrm{ differentiable a.e. with } \ \!  \|h'\|_\infty \leq 1\big\}$, 
\item $ \mathcal H_{4,\infty}:=\big\{h:\real \to \real \textrm{ four times differentiable a.e. with }  \displaystyle \max_{1\leq i \leq 4} \|h^{(i)}\|_\infty \leq 1 \big\}$. 
\end{enumerate}
Given $F$ and $G$ two random variables on a probability space $(\Omega, \mathcal F_\infty^N,\P)$, we let
$$ d_\mathcal{H}(F,G) :=\sup_{h \in \mathcal{H}} \left| \E[h(F)-h(G)] \right|
$$
denote the distance (with respect to the class of test functions $\mathcal H$) between the laws $\mathcal{L}_F$ and $\mathcal{L}_G$ of $F$ and $G$.
 In addition, 
\begin{enumerate}
\item if $ \mathcal H=\mathcal H_{W}$ we write $d_{W}$ for $d_{\mathcal{H}_{W}}$ and corresponds to the Wasserstein distance;  
\item if $ \mathcal H=\mathcal H_{4,\infty}$ we write $d_{4,\infty}$ for $d_{\mathcal H_{4,\infty}}$.
\end{enumerate}
 We also set 
\begin{enumerate}
\item $\mathcal F_{W}:=\big\{f:\real \to \real \textrm{ twice differentiable map }, \ \!  \|f'\|_\infty \leq 1, \ \!  \|f''\|_\infty \leq 2\big\}$, 
\item $\mathcal F_{4,\infty}:=\big\{f:\real \to \real \textrm{ four times differentiable map }, \ \!  \|f^{(i)}\|_\infty \leq 1/i, \ \!  i = 1,2,3,4 \big\}$, 
\end{enumerate}
 and 
$$\mathcal F_{\mathcal{H}} := \left\lbrace \begin{array}{l} \mathcal F_{W} \textrm{ if } \mathcal H = \mathcal H_{W},
  \medskip
  \\ \mathcal F_{4,\infty} \textrm{ if } \mathcal H = \mathcal H_{4,\infty}.
\end{array}\right.$$
\end{definition}

\noindent
Let ${\cal N}_{\sigma^2} \sim \mathcal{N}(0,\sigma^2)$,
let $\mathcal H$ be one of the spaces in 1.-2. above,
 and let $h$ in $\mathcal H$.
 C.M. Stein proved in \cite{Stein_method} (see also \cite[Lemma 2.6 and Section 4.8]{ChenGolsteinShao} for the $\mathcal H_{4,\infty}$ distance), that there exists a function $f_h$ in $\mathcal{F}_W$ solution to the functional Stein equation 
$$ h(x)-\E[h({\cal N}_{\sigma^2})] = \sigma^2 f_h'(x) - x f_h(x), \quad x \in \real. $$
 For $F$ a centered random variable,
 plugging $F$ in this equation and taking expectations, we get  
 $$ \left|\E[h(F)-h({\cal N}_{\sigma^2})]\right| = \left|\E[\sigma^2 f_h'(F) - F f_h(F)]\right|,
  $$
 which yields 
\begin{equation}
\label{eq:Stein}
d_{\mathcal H}(F,{\cal N}_{\sigma^2}) \leq \sup_{f \in \mathcal{F}_{\mathcal H}} \left|\E[\sigma^2 f'(F) - F f(F)]\right|.
\end{equation}
In addition, the right hand side is equal to $0$ if and only if $F\sim \mathcal{N}(0,\sigma^2)$.	

\subsubsection*{Competing interests}
The authors have no competing interests to declare that are relevant to the content of this article.

\subsubsection*{Data availability statement}
 No new data were created during the study.

\end{document}